\documentclass[10pt]{amsart}
\usepackage{amsfonts,amssymb,amscd,amsmath,enumerate,verbatim,calc}
\usepackage{color,soul}
\usepackage{mathtools}
\usepackage{graphicx}
\usepackage{booktabs}
\usepackage{tabularx}
\usepackage{array}
\usepackage{tikz}
\usepackage{url}
\usepackage{enumitem}
\usetikzlibrary{arrows.meta,positioning}
\everymath{\displaystyle}
\usepackage{fancyhdr}
\usepackage{tikz}
\usepackage{url}
\usepackage{pgfplots}       
\pgfplotsset{compat=1.17}   
\usepackage{tikz}           
\usepackage{pgfplotstable}  

\usetikzlibrary{shapes.geometric, arrows}
\tikzstyle{arrow} = [thick,->,>=stealth]
\tikzstyle{process} = [rectangle, minimum width=4cm, minimum height=2cm, text centered, text width=2.5cm, draw=green, fill=green!10]
\theoremstyle{plain}

\newtheorem{theorem}{Theorem}[section]
\newtheorem{corollary}[theorem]{Corollary}

\newtheorem{proposition}[theorem]{Proposition}

\theoremstyle{definition}
\newtheorem{definition}[theorem]{Definition}

\newtheorem{remark}[theorem]{Remark}
\newtheorem{example}[theorem]{Example}

\begin{document}
\title{Generalized Topological Structures Induced by $H^{*}$-Open Sets and Their Separation Properties}
\author{Neeraj kumar Tomar*}
\email{neer8393@gmail.com}
\address{Department of Applied Mathematics, Gautam Buddha University, Greater Noida, Uttar Pradesh 201312, India}
\thanks{*Corresponding author, Email: neer8393@gmail.com}

\author{Komal Priya}
 \email{komalpriya004@gmail.com}
\address{Department of Applied Mathematics, Gautam Buddha University, Greater Noida, Uttar Pradesh 201312, India}

\author{Amit Ujlayan}
 \email{iitk.amit@gmail.com}
\address{Department of Applied Mathematics, Gautam Buddha University, Greater Noida, Uttar Pradesh 201312, India}

\begin{abstract}
In this paper, we investigate the family of $H^{*}$-open sets from the perspective of generalized topology and examine the structural and separation properties arising from it. For a topological space
$(\mathbb{X},\tau)$, the family of all $H^{*}$-open subsets of
$\mathbb{X}$ is considered as a natural candidate for an
$H^{*}$-generated generalized topology. We establish conditions under
which this family forms a strong generalized topology and identify
additional assumptions that allow the resulting structure to be
treated as an ordinary topology. This distinction provides a
systematic framework for studying $H^{*}$-open sets without imposing
finite-intersection properties in advance.

Within this framework, we investigate the corresponding
$H^{*}$-interior and $H^{*}$-closure operators and analyze separation
properties determined by the $H^{*}$-open structure. When the
$H^{*}$-generated family forms a topology, the resulting associated
space is used to relate generalized $H^{*}$-separation properties to
their classical topological counterparts. We further examine
compactness-type properties, subspace behaviour, and preservation
under suitable mappings. Examples are included to illustrate how the
structure generated by $H^{*}$-open sets may exhibit separation
behaviour different from that of the original topology. The results
provide a unified setting for connecting $H^{*}$-open-set theory with
generalized topological structures and offer a basis for further
investigation of separation and covering properties in this framework.
\end{abstract}


\maketitle

\section*{Introduction}
\label{sec:introduction}

Generalized open and closed sets have played a significant role in
the development of general topology by providing alternative
frameworks for extending classical notions of continuity,
compactness, separation, and related topological properties. One of
the earliest contributions in this direction was the introduction of
semi-open sets by Levine \cite{levine1963semi}. Subsequently,
Njastad \cite{njȧstad1965some} introduced $\alpha$-open sets, while
Levine \cite{levine1970gen} initiated the systematic study of
generalized closed sets. These ideas stimulated the introduction of
many generalized classes of subsets and demonstrated that useful
topological information can often be obtained by replacing ordinary
open or closed sets with suitably defined generalized ones.

The influence of generalized openness is particularly visible in the
study of mappings and separation properties. Various generalized
forms of continuity and related mappings have been investigated in
\cite{arya1974strongly,jankovic1985note,decompositiondontchev1996,Noiri1994,
hatir1996decomposition}, whereas generalized
forms of regularity and normality have been considered in
\cite{singal1969almost,quasidontchev,Noiri1998,sharma2014xi,kumar2017xi,kumar2018softly}. These developments
illustrate a general principle in generalized topology: modifying
the underlying class of open sets may lead to separation and mapping
properties that differ substantially from those determined by the
original topology.

Among such generalized constructions, the theory of $h$-closed sets
provides a useful foundation for further extensions. Rodrigo, Jansi
and Theodore \cite{rodrigo2012international} studied $h$-closed sets
and related notions in topological spaces. The corresponding
$h$-closure operation subsequently provides a natural setting for
defining additional generalized closed sets. In particular, for a
topological space $(\mathbb{X},\tau)$, a subset
$\mathbb{A}\subseteq\mathbb{X}$ is called $H^{*}$-closed if
\[
h\text{-}\operatorname{Cl}(\mathbb{A})\subseteq\mathbb{U}
\]
whenever $\mathbb{A}\subseteq\mathbb{U}$ and $\mathbb{U}$ is
$hCg$-open in $\mathbb{X}$. The complement of an $H^{*}$-closed
set is called $H^{*}$-open. We write
\[
H^{*}C(\mathbb{X})
\quad\text{and}\quad
H^{*}O(\mathbb{X})
\]
for the families of all $H^{*}$-closed and $H^{*}$-open subsets of
$\mathbb{X}$, respectively.

The $H^{*}$-structure has already led to several developments in
generalized separation theory. Sharma, Sharma and Tomar
\cite{sharma2015h} investigated $H^{*}$-normal spaces, thereby
extending a classical separation concept to the setting determined
by $H^{*}$-open and $H^{*}$-closed sets. More recently, Tomar,
Sharma, Ujlayan and Zulfeqarr \cite{tomarsome2024} studied
separation axioms associated with generalized open sets and
introduced, among others, the conditions
\[
H^{*}\text{-}T_{1/2},\qquad
H^{*}\text{-}T_b,\qquad
H^{*}\text{-}T_d.
\]
They investigated several properties and relationships involving
these separation conditions. These results indicate that
$H^{*}$-open sets constitute a useful framework for reformulating
and comparing classical and generalized separation principles.

A further development was obtained by Tomar, Rani and Subadra
\cite{Tomar2026}, who introduced $\pi H^{*}$-closed sets as an
extension of the $H^{*}$-closed structure. They investigated the
fundamental properties of these sets, the corresponding
$\pi H^{*}$-open sets, and an associated separation condition. Thus,
the progression from $H^{*}$-closed sets to $\pi H^{*}$-closed sets
shows that the $H^{*}$-framework is sufficiently flexible to support
further generalized set constructions and separation properties.

Difference-set methods provide another useful approach to separation
theory. Tong \cite{Tong1982} introduced the classical $D_1$
separation axiom by means of differences of open sets, obtaining a
separation condition situated between the classical $T_0$ and
$T_1$ axioms. This approach is particularly compatible with
generalized openness, since ordinary open sets occurring in a
difference-set construction may be replaced by an appropriate
generalized class.

The preceding developments naturally raise a more fundamental
structural question. Although $H^{*}$-open sets have been used to
formulate several generalized topological notions, the family
\[
H^{*}O(\mathbb{X})
\]
should not, in general, be identified automatically with an ordinary
topology. In particular, the definition of an $H^{*}$-open set by
itself does not guarantee the finite-intersection property required
of a topology. It is therefore important to determine precisely what
type of structure is carried by $H^{*}O(\mathbb{X})$ and which
additional conditions are required for that structure to become an
ordinary topology.

This observation motivates the present work. We first consider
$H^{*}O(\mathbb{X})$ at the level of generalized topology. When
$H^{*}C(\mathbb{X})$ is closed under arbitrary intersections, the
family $H^{*}O(\mathbb{X})$ is closed under arbitrary unions and
hence determines a strong generalized topology, which we denote by
\[
\Gamma_{H^{*}}=H^{*}O(\mathbb{X}).
\]
We then identify the additional finite-intersection requirement
under which this generalized topology becomes an ordinary topology.
In that case, we write
\[
\tau_{H^{*}}=H^{*}O(\mathbb{X}).
\]
Thus, $\Gamma_{H^{*}}$ and $\tau_{H^{*}}$ represent two distinct
levels of structure and are not used interchangeably.

This distinction allows the $H^{*}$-closure and $H^{*}$-interior
operators to be studied in a systematic manner. At the
generalized-topological level, we examine their basic structural
properties under the appropriate closure assumptions. When
$\tau_{H^{*}}$ exists, we show that
\[
H^{*}\operatorname{Cl}(\mathbb{A})
=
\operatorname{Cl}_{\tau_{H^{*}}}(\mathbb{A})
\]
and
\[
H^{*}\operatorname{Int}(\mathbb{A})
=
\operatorname{Int}_{\tau_{H^{*}}}(\mathbb{A}),
\]
thereby identifying the $H^{*}$-operators with the ordinary closure
and interior operators of the associated space
$(\mathbb{X},\tau_{H^{*}})$.

The associated topology then provides a natural bridge between
generalized $H^{*}$-properties and classical topology. We use this
viewpoint to examine $H^{*}$-$T_0$, $H^{*}$-$T_1$, and
$H^{*}$-$T_2$ separation, as well as $H^{*}$-difference separation.
Higher separation properties, including regularity, complete
regularity, and normality, are also considered under the hypotheses
required for the associated topology. In addition, we study
$H^{*}$-compactness and $H^{*}$-irresolute mappings and relate them
to their classical counterparts in the associated spaces.

A further objective is to compare the original topology with the
structure induced by $H^{*}$-open sets. Suitable finite and infinite
examples show that the associated topology, when it exists, may
differ considerably from the original topology. In particular, the
$H^{*}$-generated topology may be discrete even when the original
space fails to satisfy the $T_0$ separation axiom. These examples
make clear that the passage from $\tau$ to the $H^{*}$-generated
structure is not merely a change of notation but may substantially
alter the separation behaviour of the space.

Accordingly, the present study distinguishes the three structures
\[
(\mathbb{X},\tau),\qquad
(\mathbb{X},\Gamma_{H^{*}}),\qquad
(\mathbb{X},\tau_{H^{*}}).
\]
Here, $(\mathbb{X},\tau)$ is the original topological space;
$\Gamma_{H^{*}}$ denotes the strong generalized topology determined
by $H^{*}$-open sets under the required arbitrary-union condition;
and $\tau_{H^{*}}$ denotes the associated ordinary topology whenever
the additional finite-intersection condition is satisfied. This
three-level viewpoint provides a unified framework for identifying
which properties belong intrinsically to the generalized
$H^{*}$-structure and which can be interpreted through ordinary
topological methods.

The remainder of the paper is organized as follows.
Section~\ref{sec:preliminaries} recalls the definitions and
preliminary results needed in the sequel.
Section~\ref{sec:Hstar-generalized-topology} studies the structure
induced by $H^{*}$-open sets and gives conditions under which it is
a strong generalized topology or an ordinary topology.
Section~\ref{sec:operators} investigates the $H^{*}$-closure and
$H^{*}$-interior operators.
Section~\ref{sec:separation} develops the relevant separation
properties, and Section~\ref{sec:higher-separation} treats higher
separation conditions through the associated topology.
Section~\ref{sec:examples} presents examples illustrating the
relationship between the original and associated structures.
Section~\ref{sec:compactness} discusses $H^{*}$-compactness, while
Section~\ref{sec:mappings} is devoted to mappings and preservation
properties.
Section~\ref{sec:comparison} relates the resulting theory to the
generalized-topological viewpoint.
Finally, Section~\ref{sec:conclusion} summarizes the main conclusions
and indicates directions for further investigation.

\section{Preliminaries}
\label{sec:preliminaries}

Throughout this paper, $(\mathbb{X},\tau)$ denotes a topological
space. For a subset $\mathbb{A}\subseteq\mathbb{X}$, the ordinary
closure and interior of $\mathbb{A}$ are denoted by
$\operatorname{cl}(\mathbb{A})$ and
$\operatorname{int}(\mathbb{A})$, respectively.

We first recall some generalized forms of closed sets that are
required for the definition and subsequent development of
$H^{*}$-closed sets. The terminology used below follows the
corresponding notions available in the literature
\cite{levine1963semi,levine1970gen,njȧstad1965some,
rodrigo2012international}.

\begin{definition}
\label{def:generalized-closed-sets}
Let $\mathbb{A}\subseteq\mathbb{X}$. Then $\mathbb{A}$ is said to be:
\begin{enumerate}[label=(\roman*)]

    \item \emph{regular-closed} \cite{stone1937applications} if
    \[
    \mathbb{A}
    =
    \operatorname{cl}
    \bigl(\operatorname{int}(\mathbb{A})\bigr);
    \]

    \item \emph{semi-closed} \cite{levine1963semi} if
    \[
    \operatorname{int}
    \bigl(\operatorname{cl}(\mathbb{A})\bigr)
    \subseteq \mathbb{A};
    \]

    \item \emph{$w$-closed} \cite{Sundaram2000} if
    \[
    \operatorname{cl}(\mathbb{A})\subseteq\mathbb{U},
    \]
    whenever $\mathbb{A}\subseteq\mathbb{U}$ and
    $\mathbb{U}$ is semi-open in $\mathbb{X}$;

    \item \emph{$\alpha$-closed} \cite{njȧstad1965some} if
    \[
    \operatorname{cl}
    \bigl(
        \operatorname{int}
        (\operatorname{cl}(\mathbb{A}))
    \bigr)
    \subseteq\mathbb{A};
    \]

    \item an \emph{$\alpha^{*}$-set} \cite{hatir1996decomposition} if
    \[
    \operatorname{int}
    \bigl(
        \operatorname{cl}
        (\operatorname{int}(\mathbb{A}))
    \bigr)
    =
    \operatorname{int}(\mathbb{A});
    \]

    \item a \emph{$C$-set} \cite{hatir1996decomposition} if
    \[
    \mathbb{A}=\mathbb{U}\cap\mathbb{V},
    \]
    where $\mathbb{U}$ is open and $\mathbb{V}$ is an
    $\alpha^{*}$-set in $\mathbb{X}$;

    \item \emph{$h$-closed} \cite{rodrigo2012international} if
    \[
    s\text{-}\operatorname{cl}(\mathbb{A})
    \subseteq\mathbb{U},
    \]
    whenever $\mathbb{A}\subseteq\mathbb{U}$ and
    $\mathbb{U}$ is $w$-open in $\mathbb{X}$;

    \item \emph{$gh$-closed} \cite{rodrigo2012international} if
    \[
    h\text{-}\operatorname{cl}(\mathbb{A})
    \subseteq\mathbb{U},
    \]
    whenever $\mathbb{A}\subseteq\mathbb{U}$ and
    $\mathbb{U}$ is $h$-open in $\mathbb{X}$;

    \item \emph{regular-$h$-open} \cite{rodrigo2012international} if there exists a regular-open
    set $\mathbb{U}$ such that
    \[
    \mathbb{U}
    \subseteq\mathbb{A}
    \subseteq
    h\text{-}\operatorname{cl}(\mathbb{U});
    \]

    \item \emph{$rgh$-closed} \cite{rodrigo2012international} if
    \[
    h\text{-}\operatorname{cl}(\mathbb{A})
    \subseteq\mathbb{U},
    \]
    whenever $\mathbb{A}\subseteq\mathbb{U}$ and
    $\mathbb{U}$ is regular-$h$-open in $\mathbb{X}$;

    \item \emph{$hCg$-closed} \cite{rodrigo2012international} if
    \[
    h\text{-}\operatorname{cl}(\mathbb{A})
    \subseteq\mathbb{U},
    \]
    whenever $\mathbb{A}\subseteq\mathbb{U}$ and
    $\mathbb{U}$ is a $C$-set in $\mathbb{X}$.

\end{enumerate}
\end{definition}

The complements of regular-closed, semi-closed, $w$-closed,
$\alpha$-closed, $h$-closed, $gh$-closed, $rgh$-closed, and
$hCg$-closed sets are called regular-open, semi-open, $w$-open,
$\alpha$-open, $h$-open, $gh$-open, $rgh$-open, and $hCg$-open
sets, respectively.

For a subset $\mathbb{A}\subseteq\mathbb{X}$, its $h$-closure \cite{rodrigo2012international} and
semi-closure are denoted by
\[
h\text{-}\operatorname{cl}(\mathbb{A})
\quad\text{and}\quad
s\text{-}\operatorname{cl}(\mathbb{A}),
\]
respectively. Similarly, the corresponding interiors are denoted by
\[
h\text{-}\operatorname{int}(\mathbb{A})
\quad\text{and}\quad
s\text{-}\operatorname{int}(\mathbb{A}).
\]

\subsection{$H^{*}$-Closed and $H^{*}$-Open Sets}\cite{Tomar2026}
\label{subsec:Hstar-sets}

We next recall the notion of an $H^{*}$-closed set, which forms the
basis of the structure investigated in this paper.

\begin{definition}
\label{def:Hstar-closed}
Let $\mathbb{A}\subseteq\mathbb{X}$. Then $\mathbb{A}$ is called
\emph{$H^{*}$-closed} \cite{Tomar2026} if
\[
h\text{-}\operatorname{cl}(\mathbb{A})
\subseteq\mathbb{U}
\]
whenever
\[
\mathbb{A}\subseteq\mathbb{U}
\]
and $\mathbb{U}$ is $hCg$-open in $\mathbb{X}$.

A subset $\mathbb{A}$ is called \emph{$H^{*}$-open} if
$\mathbb{X}\setminus\mathbb{A}$ is $H^{*}$-closed.
\end{definition}

We denote the families of all $H^{*}$-closed and $H^{*}$-open
subsets of $\mathbb{X}$ by
\[
H^{*}C(\mathbb{X})
=
\left\{
\mathbb{F}\subseteq\mathbb{X}:
\mathbb{F}\text{ is }H^{*}\text{-closed}
\right\}
\]
and
\[
H^{*}O(\mathbb{X})
=
\left\{
\mathbb{U}\subseteq\mathbb{X}:
\mathbb{U}\text{ is }H^{*}\text{-open}
\right\},
\]
respectively.

The following elementary observation will be used when constructing
the generalized topology generated by $H^{*}$-open sets.

\begin{proposition}
\label{prop:empty-X-Hstar}
For every topological space $(\mathbb{X},\tau)$,
\[
\varnothing,\mathbb{X}\in H^{*}C(\mathbb{X})
\]
and consequently
\[
\varnothing,\mathbb{X}\in H^{*}O(\mathbb{X}).
\]
\end{proposition}

\begin{proof}
Since
\[
h\text{-}\operatorname{cl}(\varnothing)=\varnothing,
\]
the defining condition of
Definition~\ref{def:Hstar-closed} is satisfied for
$\varnothing$. Hence
\[
\varnothing\in H^{*}C(\mathbb{X}).
\]

Now consider $\mathbb{X}$. If
\[
\mathbb{X}\subseteq\mathbb{U}\subseteq\mathbb{X},
\]
then necessarily $\mathbb{U}=\mathbb{X}$. Since
\[
h\text{-}\operatorname{cl}(\mathbb{X})=\mathbb{X},
\]
it follows that
\[
h\text{-}\operatorname{cl}(\mathbb{X})
\subseteq\mathbb{U}.
\]
Thus $\mathbb{X}$ is $H^{*}$-closed. The corresponding statements
for $H^{*}$-open sets follow immediately by taking complements.
\end{proof}

\subsection{$H^{*}$-Closure and $H^{*}$-Interior}\cite{Tomar2026}
\label{subsec:Hstar-operators-prelim}

\begin{definition}
\label{def:Hstar-operators}
Let $\mathbb{A}\subseteq\mathbb{X}$. The
\emph{$H^{*}$-closure} of $\mathbb{A}$ is defined by
\begin{equation}
H^{*}\operatorname{Cl}(\mathbb{A})
=
\bigcap
\left\{
\mathbb{F}\in H^{*}C(\mathbb{X}):
\mathbb{A}\subseteq\mathbb{F}
\right\},
\label{eq:Hstar-closure}
\end{equation}
and the \emph{$H^{*}$-interior} of $\mathbb{A}$ is defined by
\begin{equation}
H^{*}\operatorname{Int}(\mathbb{A})
=
\bigcup
\left\{
\mathbb{U}\in H^{*}O(\mathbb{X}):
\mathbb{U}\subseteq\mathbb{A}
\right\}.
\label{eq:Hstar-interior}
\end{equation}
\end{definition}

Notice that Definition~\ref{def:Hstar-operators} does not require
$H^{*}O(\mathbb{X})$ to be an ordinary topology. The operators are
defined directly from the families of $H^{*}$-closed and
$H^{*}$-open sets.

\begin{proposition}
\label{prop:Hstar-basic-operators}
For all $\mathbb{A},\mathbb{B}\subseteq\mathbb{X}$, the following
statements hold:
\begin{enumerate}[label=(\roman*)]

    \item
    \[
    \mathbb{A}
    \subseteq
    H^{*}\operatorname{Cl}(\mathbb{A});
    \]

    \item if $\mathbb{A}\subseteq\mathbb{B}$, then
    \[
    H^{*}\operatorname{Cl}(\mathbb{A})
    \subseteq
    H^{*}\operatorname{Cl}(\mathbb{B});
    \]

    \item
    \[
    H^{*}\operatorname{Int}(\mathbb{A})
    \subseteq
    \mathbb{A};
    \]

    \item if $\mathbb{A}\subseteq\mathbb{B}$, then
    \[
    H^{*}\operatorname{Int}(\mathbb{A})
    \subseteq
    H^{*}\operatorname{Int}(\mathbb{B}).
    \]

\end{enumerate}
\end{proposition}

\begin{proof}
Statement (i) follows directly from
\eqref{eq:Hstar-closure}, since every $H^{*}$-closed set occurring
in the defining intersection contains $\mathbb{A}$.

Suppose that $\mathbb{A}\subseteq\mathbb{B}$. Every
$H^{*}$-closed set containing $\mathbb{B}$ also contains
$\mathbb{A}$. Hence the family of $H^{*}$-closed supersets used to
define $H^{*}\operatorname{Cl}(\mathbb{B})$ is contained in the
corresponding family used to define
$H^{*}\operatorname{Cl}(\mathbb{A})$. Therefore,
\[
H^{*}\operatorname{Cl}(\mathbb{A})
\subseteq
H^{*}\operatorname{Cl}(\mathbb{B}),
\]
which proves (ii).

Statement (iii) follows from \eqref{eq:Hstar-interior}, since every
member of the defining union is contained in $\mathbb{A}$.

Finally, if $\mathbb{A}\subseteq\mathbb{B}$, every $H^{*}$-open
subset of $\mathbb{A}$ is also an $H^{*}$-open subset of
$\mathbb{B}$. Consequently,
\[
H^{*}\operatorname{Int}(\mathbb{A})
\subseteq
H^{*}\operatorname{Int}(\mathbb{B}),
\]
and (iv) follows.
\end{proof}

\begin{proposition}
\label{prop:Hstar-duality}
For every $\mathbb{A}\subseteq\mathbb{X}$,
\begin{equation}
H^{*}\operatorname{Int}(\mathbb{A})
=
\mathbb{X}\setminus
H^{*}\operatorname{Cl}
(\mathbb{X}\setminus\mathbb{A}).
\label{eq:Hstar-int-cl-duality}
\end{equation}
Equivalently,
\begin{equation}
H^{*}\operatorname{Cl}(\mathbb{A})
=
\mathbb{X}\setminus
H^{*}\operatorname{Int}
(\mathbb{X}\setminus\mathbb{A}).
\label{eq:Hstar-cl-int-duality}
\end{equation}
\end{proposition}

\begin{proof}
By Definition~\ref{def:Hstar-operators},
\[
H^{*}\operatorname{Int}(\mathbb{A})
=
\bigcup
\left\{
\mathbb{U}\in H^{*}O(\mathbb{X}):
\mathbb{U}\subseteq\mathbb{A}
\right\}.
\]
Since $\mathbb{U}$ is $H^{*}$-open if and only if
$\mathbb{X}\setminus\mathbb{U}$ is $H^{*}$-closed, the condition
$\mathbb{U}\subseteq\mathbb{A}$ is equivalent to
\[
\mathbb{X}\setminus\mathbb{A}
\subseteq
\mathbb{X}\setminus\mathbb{U}.
\]
Hence, by De Morgan's law,
\[
H^{*}\operatorname{Int}(\mathbb{A})
=
\mathbb{X}\setminus
\bigcap
\left\{
\mathbb{F}\in H^{*}C(\mathbb{X}):
\mathbb{X}\setminus\mathbb{A}\subseteq\mathbb{F}
\right\}.
\]
The intersection on the right-hand side is precisely
\[
H^{*}\operatorname{Cl}
(\mathbb{X}\setminus\mathbb{A}),
\]
which proves \eqref{eq:Hstar-int-cl-duality}. The second identity
follows by replacing $\mathbb{A}$ with
$\mathbb{X}\setminus\mathbb{A}$.
\end{proof}

\subsection{Generalized Topological Spaces}
\label{subsec:generalized-topology}

We finally recall the generalized-topological terminology required
in Section~\ref{sec:Hstar-generalized-topology}. The notion of a
generalized topology was introduced by Cs\'asz\'ar and has since
been used extensively in the study of generalized open sets and
separation properties.

\begin{definition}
\label{def:generalized-topology}
Let $\mathbb{X}$ be a nonempty set. A family
$\Gamma\subseteq\mathcal{P}(\mathbb{X})$ is called a
\emph{generalized topology} on $\mathbb{X}$ if:
\begin{enumerate}[label=(\roman*)]
    \item $\varnothing\in\Gamma$;
    \item for every family
    $\{\mathbb{U}_{\lambda}:\lambda\in\Lambda\}\subseteq\Gamma$,
    \[
    \bigcup_{\lambda\in\Lambda}\mathbb{U}_{\lambda}\in\Gamma.
    \]
\end{enumerate}
The pair $(\mathbb{X},\Gamma)$ is called a
\emph{generalized topological space}.
\end{definition}

\begin{definition}
\label{def:strong-generalized-topology}
A generalized topology $\Gamma$ on $\mathbb{X}$ is called
\emph{strong} if
\[
\mathbb{X}\in\Gamma.
\]
In this case, $(\mathbb{X},\Gamma)$ is called a
\emph{strong generalized topological space}.
\end{definition}

Unlike an ordinary topology, a generalized topology is not required
to be closed under finite intersections. Consequently, even a
strong generalized topology need not be an ordinary topology. This
distinction is fundamental to the $H^{*}$-generated structure
studied in Section~\ref{sec:Hstar-generalized-topology}.

\begin{definition}
\label{def:Gamma-interior}
Let $(\mathbb{X},\Gamma)$ be a generalized topological space and
let $\mathbb{A}\subseteq\mathbb{X}$. The $\Gamma$-interior of
$\mathbb{A}$ is defined by
\[
i_{\Gamma}(\mathbb{A})
=
\bigcup
\left\{
\mathbb{U}\in\Gamma:
\mathbb{U}\subseteq\mathbb{A}
\right\}.
\]
\end{definition}

\begin{definition}
\label{def:Gamma-closure}
Let $(\mathbb{X},\Gamma)$ be a generalized topological space. A
subset $\mathbb{F}\subseteq\mathbb{X}$ is called
$\Gamma$-closed if
\[
\mathbb{X}\setminus\mathbb{F}\in\Gamma.
\]
For $\mathbb{A}\subseteq\mathbb{X}$, the $\Gamma$-closure of
$\mathbb{A}$ is defined by
\[
c_{\Gamma}(\mathbb{A})
=
\bigcap
\left\{
\mathbb{F}\subseteq\mathbb{X}:
\mathbb{A}\subseteq\mathbb{F}
\text{ and }\mathbb{F}\text{ is }\Gamma\text{-closed}
\right\}.
\]
\end{definition}

\begin{remark}
\label{rem:Hstar-not-automatic-topology}
The family $H^{*}O(\mathbb{X})$ should not be identified with an
ordinary topology without additional assumptions.
Proposition~\ref{prop:empty-X-Hstar} guarantees only that
\[
\varnothing,\mathbb{X}\in H^{*}O(\mathbb{X}).
\]
Closure under arbitrary unions and closure under finite
intersections are separate structural requirements. In particular,
finite intersections of $H^{*}$-open sets need not, in general, be
$H^{*}$-open. Accordingly, the generalized-topological structure
and the associated ordinary topology, whenever the latter exists,
will be distinguished throughout the paper.
\end{remark}

\section{The Generalized Topology Induced by $H^{*}$-Open Sets}
\label{sec:Hstar-generalized-topology}

In this section, we investigate the structure generated by the family
of $H^{*}$-open subsets of a topological space
$(\mathbb{X},\tau)$. As observed in
Remark~\ref{rem:Hstar-not-automatic-topology}, the family
$H^{*}O(\mathbb{X})$ need not be an ordinary topology. We first
identify a sufficient condition under which this family determines a
strong generalized topology and then characterize when the resulting
generalized topology is an ordinary topology.

\begin{theorem}
\label{thm:Hstar-strong-generalized-topology}
Let $(\mathbb{X},\tau)$ be a topological space. Suppose that
$H^{*}C(\mathbb{X})$ is closed under arbitrary intersections. Then
\[
\Gamma_{H^{*}}=H^{*}O(\mathbb{X})
\]
is a strong generalized topology on $\mathbb{X}$.
\end{theorem}

\begin{proof}
By Proposition~\ref{prop:empty-X-Hstar},
\[
\varnothing,\mathbb{X}\in H^{*}O(\mathbb{X}).
\]
It remains to verify closure under arbitrary unions.

Let
\[
\{\mathbb{U}_{\lambda}:\lambda\in\Lambda\}
\subseteq H^{*}O(\mathbb{X})
\]
be an arbitrary family. For every $\lambda\in\Lambda$,
\[
\mathbb{X}\setminus\mathbb{U}_{\lambda}
\in H^{*}C(\mathbb{X}).
\]
Since $H^{*}C(\mathbb{X})$ is closed under arbitrary intersections,
\[
\bigcap_{\lambda\in\Lambda}
\bigl(\mathbb{X}\setminus\mathbb{U}_{\lambda}\bigr)
\in H^{*}C(\mathbb{X}).
\]
By De Morgan's law,
\[
\mathbb{X}\setminus
\bigcup_{\lambda\in\Lambda}\mathbb{U}_{\lambda}
=
\bigcap_{\lambda\in\Lambda}
\bigl(\mathbb{X}\setminus\mathbb{U}_{\lambda}\bigr).
\]
Hence
\[
\mathbb{X}\setminus
\bigcup_{\lambda\in\Lambda}\mathbb{U}_{\lambda}
\in H^{*}C(\mathbb{X}),
\]
and therefore
\[
\bigcup_{\lambda\in\Lambda}\mathbb{U}_{\lambda}
\in H^{*}O(\mathbb{X}).
\]

Thus $H^{*}O(\mathbb{X})$ contains $\varnothing$ and
$\mathbb{X}$ and is closed under arbitrary unions. Consequently,
\[
\Gamma_{H^{*}}=H^{*}O(\mathbb{X})
\]
is a strong generalized topology on $\mathbb{X}$.
\end{proof}

Whenever the hypothesis of
Theorem~\ref{thm:Hstar-strong-generalized-topology} is satisfied, we
denote the resulting generalized topological space by
\[
(\mathbb{X},\Gamma_{H^{*}}),
\qquad
\Gamma_{H^{*}}=H^{*}O(\mathbb{X}),
\]
and refer to it as the \emph{$H^{*}$-generated generalized
topological space} associated with $(\mathbb{X},\tau)$.

\begin{proposition}
\label{prop:Gamma-Hstar-closed}
Under the hypothesis of
Theorem~\ref{thm:Hstar-strong-generalized-topology}, a subset
$\mathbb{F}\subseteq\mathbb{X}$ is
$\Gamma_{H^{*}}$-closed if and only if it is $H^{*}$-closed.
Consequently,
\[
\mathcal{C}_{\Gamma_{H^{*}}}
=
H^{*}C(\mathbb{X}),
\]
where $\mathcal{C}_{\Gamma_{H^{*}}}$ denotes the family of all
$\Gamma_{H^{*}}$-closed subsets of $\mathbb{X}$.
\end{proposition}

\begin{proof}
By definition, $\mathbb{F}$ is $\Gamma_{H^{*}}$-closed if and only if
\[
\mathbb{X}\setminus\mathbb{F}\in\Gamma_{H^{*}}.
\]
Since
\[
\Gamma_{H^{*}}=H^{*}O(\mathbb{X}),
\]
this is equivalent to
\[
\mathbb{X}\setminus\mathbb{F}\in H^{*}O(\mathbb{X}).
\]
By the definition of an $H^{*}$-open set, the latter holds if and
only if
\[
\mathbb{F}\in H^{*}C(\mathbb{X}).
\]
Hence the two families of closed sets coincide.
\end{proof}

\begin{theorem}
\label{thm:Gamma-Hstar-operators}
Let $(\mathbb{X},\tau)$ satisfy the hypothesis of
Theorem~\ref{thm:Hstar-strong-generalized-topology}. Then, for every
$\mathbb{A}\subseteq\mathbb{X}$,
\[
c_{\Gamma_{H^{*}}}(\mathbb{A})
=
H^{*}\operatorname{Cl}(\mathbb{A})
\]
and
\[
i_{\Gamma_{H^{*}}}(\mathbb{A})
=
H^{*}\operatorname{Int}(\mathbb{A}).
\]
\end{theorem}

\begin{proof}
By Definition~\ref{def:Gamma-closure},
\[
c_{\Gamma_{H^{*}}}(\mathbb{A})
=
\bigcap
\left\{
\mathbb{F}\subseteq\mathbb{X}:
\mathbb{A}\subseteq\mathbb{F}
\text{ and }
\mathbb{F}\text{ is }\Gamma_{H^{*}}\text{-closed}
\right\}.
\]
By Proposition~\ref{prop:Gamma-Hstar-closed}, the
$\Gamma_{H^{*}}$-closed sets are precisely the $H^{*}$-closed sets.
Therefore,
\[
c_{\Gamma_{H^{*}}}(\mathbb{A})
=
\bigcap
\left\{
\mathbb{F}\in H^{*}C(\mathbb{X}):
\mathbb{A}\subseteq\mathbb{F}
\right\}
=
H^{*}\operatorname{Cl}(\mathbb{A}).
\]

Similarly,
\[
i_{\Gamma_{H^{*}}}(\mathbb{A})
=
\bigcup
\left\{
\mathbb{U}\in\Gamma_{H^{*}}:
\mathbb{U}\subseteq\mathbb{A}
\right\}.
\]
Since $\Gamma_{H^{*}}=H^{*}O(\mathbb{X})$, we obtain
\[
i_{\Gamma_{H^{*}}}(\mathbb{A})
=
\bigcup
\left\{
\mathbb{U}\in H^{*}O(\mathbb{X}):
\mathbb{U}\subseteq\mathbb{A}
\right\}
=
H^{*}\operatorname{Int}(\mathbb{A}).
\]
\end{proof}

\begin{corollary}
\label{cor:Hstar-closure-closed}
Under the hypothesis of
Theorem~\ref{thm:Hstar-strong-generalized-topology}, for every
$\mathbb{A}\subseteq\mathbb{X}$,
\[
H^{*}\operatorname{Cl}(\mathbb{A})
\in H^{*}C(\mathbb{X})
\]
and
\[
H^{*}\operatorname{Int}(\mathbb{A})
\in H^{*}O(\mathbb{X}).
\]
\end{corollary}

\begin{proof}
Since $H^{*}C(\mathbb{X})$ is closed under arbitrary intersections,
the intersection defining
$H^{*}\operatorname{Cl}(\mathbb{A})$ is $H^{*}$-closed.

Moreover, by
Theorem~\ref{thm:Hstar-strong-generalized-topology},
$H^{*}O(\mathbb{X})$ is closed under arbitrary unions. Hence the
union defining $H^{*}\operatorname{Int}(\mathbb{A})$ is
$H^{*}$-open.
\end{proof}

\begin{corollary}
\label{cor:Hstar-fixed-points}
Under the hypothesis of
Theorem~\ref{thm:Hstar-strong-generalized-topology}, for every
$\mathbb{A}\subseteq\mathbb{X}$:
\begin{enumerate}[label=(\roman*)]
    \item $\mathbb{A}$ is $H^{*}$-closed if and only if
    \[
    H^{*}\operatorname{Cl}(\mathbb{A})=\mathbb{A};
    \]

    \item $\mathbb{A}$ is $H^{*}$-open if and only if
    \[
    H^{*}\operatorname{Int}(\mathbb{A})=\mathbb{A}.
    \]
\end{enumerate}
\end{corollary}

\begin{proof}
Suppose first that $\mathbb{A}$ is $H^{*}$-closed. Since
$\mathbb{A}$ itself occurs among the $H^{*}$-closed supersets in the
definition of $H^{*}\operatorname{Cl}(\mathbb{A})$,
\[
H^{*}\operatorname{Cl}(\mathbb{A})\subseteq\mathbb{A}.
\]
By Proposition~\ref{prop:Hstar-basic-operators},
\[
\mathbb{A}\subseteq
H^{*}\operatorname{Cl}(\mathbb{A}),
\]
and therefore
\[
H^{*}\operatorname{Cl}(\mathbb{A})=\mathbb{A}.
\]

Conversely, if
\[
H^{*}\operatorname{Cl}(\mathbb{A})=\mathbb{A},
\]
then Corollary~\ref{cor:Hstar-closure-closed} implies that
$\mathbb{A}$ is $H^{*}$-closed.

The proof of (ii) is dual.
\end{proof}

\subsection{When $\Gamma_{H^{*}}$ Becomes an Ordinary Topology}
\label{subsec:Hstar-ordinary-topology}

We now determine when the strong generalized topology
$\Gamma_{H^{*}}=H^{*}O(\mathbb{X})$ becomes an ordinary topology.

\begin{theorem}
\label{thm:Hstar-topology-characterization}
Let $(\mathbb{X},\tau)$ be a topological space such that
$H^{*}C(\mathbb{X})$ is closed under arbitrary intersections.
Then the following statements are equivalent:
\begin{enumerate}[label=(\roman*)]

    \item $\Gamma_{H^{*}}=H^{*}O(\mathbb{X})$ is an ordinary
    topology on $\mathbb{X}$;

    \item $H^{*}O(\mathbb{X})$ is closed under finite intersections;

    \item $H^{*}C(\mathbb{X})$ is closed under finite unions.

\end{enumerate}
\end{theorem}

\begin{proof}
By Theorem~\ref{thm:Hstar-strong-generalized-topology},
$\Gamma_{H^{*}}=H^{*}O(\mathbb{X})$ contains
$\varnothing$ and $\mathbb{X}$ and is closed under arbitrary unions.
Therefore, it is an ordinary topology precisely when it is also
closed under finite intersections. Thus (i) and (ii) are equivalent.

Assume (ii), and let
$\mathbb{F}_{1},\mathbb{F}_{2}\in H^{*}C(\mathbb{X})$. Then
\[
\mathbb{X}\setminus\mathbb{F}_{1},
\quad
\mathbb{X}\setminus\mathbb{F}_{2}
\in H^{*}O(\mathbb{X}).
\]
By (ii),
\[
(\mathbb{X}\setminus\mathbb{F}_{1})
\cap
(\mathbb{X}\setminus\mathbb{F}_{2})
\in H^{*}O(\mathbb{X}).
\]
Since
\[
(\mathbb{X}\setminus\mathbb{F}_{1})
\cap
(\mathbb{X}\setminus\mathbb{F}_{2})
=
\mathbb{X}\setminus
(\mathbb{F}_{1}\cup\mathbb{F}_{2}),
\]
we obtain
\[
\mathbb{F}_{1}\cup\mathbb{F}_{2}
\in H^{*}C(\mathbb{X}).
\]
Hence (iii) holds.

Conversely, assume (iii), and let
$\mathbb{U}_{1},\mathbb{U}_{2}\in H^{*}O(\mathbb{X})$. Then
\[
\mathbb{X}\setminus\mathbb{U}_{1},
\quad
\mathbb{X}\setminus\mathbb{U}_{2}
\in H^{*}C(\mathbb{X}).
\]
By (iii),
\[
(\mathbb{X}\setminus\mathbb{U}_{1})
\cup
(\mathbb{X}\setminus\mathbb{U}_{2})
\in H^{*}C(\mathbb{X}).
\]
Using De Morgan's law,
\[
(\mathbb{X}\setminus\mathbb{U}_{1})
\cup
(\mathbb{X}\setminus\mathbb{U}_{2})
=
\mathbb{X}\setminus
(\mathbb{U}_{1}\cap\mathbb{U}_{2}).
\]
Therefore,
\[
\mathbb{U}_{1}\cap\mathbb{U}_{2}
\in H^{*}O(\mathbb{X}),
\]
and (ii) follows.
\end{proof}

\begin{definition}
\label{def:associated-Hstar-topology}
Whenever the equivalent conditions of
Theorem~\ref{thm:Hstar-topology-characterization} are satisfied, we
write
\[
\tau_{H^{*}}=H^{*}O(\mathbb{X})
\]
and call $\tau_{H^{*}}$ the \emph{topology induced by the
$H^{*}$-open sets}. The space
\[
(\mathbb{X},\tau_{H^{*}})
\]
will be referred to as the \emph{associated $H^{*}$-topological
space}.
\end{definition}

\begin{remark}
\label{rem:Gamma-vs-tau}
The symbols $\Gamma_{H^{*}}$ and $\tau_{H^{*}}$ have distinct
meanings. The notation
\[
\Gamma_{H^{*}}=H^{*}O(\mathbb{X})
\]
is used when the hypotheses establishing a strong generalized
topology are satisfied, whereas
\[
\tau_{H^{*}}=H^{*}O(\mathbb{X})
\]
is used only when the additional finite-intersection condition of
Theorem~\ref{thm:Hstar-topology-characterization} is satisfied.
Thus, the existence of $\Gamma_{H^{*}}$ does not by itself imply the
existence of the associated ordinary topology $\tau_{H^{*}}$.
\end{remark}

\subsection{The Family of $H^{*}$-Open Sets Need Not Form an Ordinary Topology}
\label{subsec:Hstar-counterexample}

The finite-intersection condition in
Theorem~\ref{thm:Hstar-topology-characterization} is essential for
the existence of the associated ordinary topology. The following
example illustrates that $H^{*}$-open sets need not be closed under
finite intersections.

\begin{example}
\label{ex:Hstar-not-topology}
Let
\[
\mathbb{X}=\{k,l,m,n\}
\]
and consider the topology
\[
\tau=
\{
\varnothing,\mathbb{X},
\{k\},\{l\},\{k,l\},\{k,m\},\{k,l,m\}
\}.
\]
For this space, the family of $H^{*}$-closed sets is
\[
\begin{aligned}
H^{*}C(\mathbb{X})
=
\{&
\varnothing,\mathbb{X},
\{k\},\{l\},\{m\},\{n\},
\{k,m\},\{k,n\},\\
&
\{l,m\},\{l,n\},\{m,n\},
\{k,l,n\},\{k,m,n\},\{l,m,n\}
\}.
\end{aligned}
\]

In particular,
\[
\{k\},\{l\}\in H^{*}C(\mathbb{X}),
\]
whereas
\[
\{k,l\}\notin H^{*}C(\mathbb{X}).
\]
Thus $H^{*}C(\mathbb{X})$ is not closed under finite unions.

Equivalently,
\[
\mathbb{X}\setminus\{k\}=\{l,m,n\}
\]
and
\[
\mathbb{X}\setminus\{l\}=\{k,m,n\}
\]
belong to $H^{*}O(\mathbb{X})$. However,
\[
\{l,m,n\}\cap\{k,m,n\}
=
\{m,n\},
\]
while
\[
\mathbb{X}\setminus\{m,n\}=\{k,l\}
\notin H^{*}C(\mathbb{X}).
\]
Therefore,
\[
\{m,n\}\notin H^{*}O(\mathbb{X}).
\]

Hence $H^{*}O(\mathbb{X})$ is not closed under finite intersections
and therefore does not form an ordinary topology on $\mathbb{X}$.
\end{example}

\begin{remark}
\label{rem:need-generalized-topology}
Example~\ref{ex:Hstar-not-topology} shows that the family of
$H^{*}$-open sets cannot, in general, be treated as an ordinary
topology. Therefore, the generalized-topological structure
$\Gamma_{H^{*}}$ is considered only under the hypothesis of
Theorem~\ref{thm:Hstar-strong-generalized-topology}, while the
ordinary topology $\tau_{H^{*}}$ is introduced only when the
additional equivalent conditions of
Theorem~\ref{thm:Hstar-topology-characterization} are satisfied.
\end{remark}

\section{Properties of $H^{*}$-Closure and $H^{*}$-Interior}
\label{sec:operators}

Throughout this section, $(\mathbb{X},\tau)$ denotes a topological
space such that $H^{*}C(\mathbb{X})$ is closed under arbitrary
intersections. Consequently, by
Theorem~\ref{thm:Hstar-strong-generalized-topology},
\[
\Gamma_{H^{*}}=H^{*}O(\mathbb{X})
\]
is a strong generalized topology on $\mathbb{X}$.

Moreover, by Theorem~\ref{thm:Gamma-Hstar-operators}, the operators
$H^{*}\operatorname{Cl}$ and $H^{*}\operatorname{Int}$ coincide
with the closure and interior operators associated with
$\Gamma_{H^{*}}$. We now record their principal structural
properties.

\subsection{Basic Operator Properties}

\begin{theorem}
\label{thm:Hstar-closure-properties}
For all $\mathbb{A},\mathbb{B}\subseteq\mathbb{X}$, the
$H^{*}$-closure operator satisfies:
\begin{enumerate}[label=(\roman*)]

    \item
    \[
    H^{*}\operatorname{Cl}(\varnothing)=\varnothing;
    \]

    \item
    \[
    H^{*}\operatorname{Cl}(\mathbb{X})=\mathbb{X};
    \]

    \item
    \[
    \mathbb{A}
    \subseteq
    H^{*}\operatorname{Cl}(\mathbb{A});
    \]

    \item if $\mathbb{A}\subseteq\mathbb{B}$, then
    \[
    H^{*}\operatorname{Cl}(\mathbb{A})
    \subseteq
    H^{*}\operatorname{Cl}(\mathbb{B});
    \]

    \item
    \[
    H^{*}\operatorname{Cl}
    \bigl(H^{*}\operatorname{Cl}(\mathbb{A})\bigr)
    =
    H^{*}\operatorname{Cl}(\mathbb{A}).
    \]

\end{enumerate}
\end{theorem}

\begin{proof}
Since $\varnothing$ is $H^{*}$-closed,
\[
H^{*}\operatorname{Cl}(\varnothing)
\subseteq\varnothing,
\]
and hence
\[
H^{*}\operatorname{Cl}(\varnothing)=\varnothing.
\]

Since $\mathbb{X}$ is $H^{*}$-closed and is the only subset of
$\mathbb{X}$ containing $\mathbb{X}$, we obtain
\[
H^{*}\operatorname{Cl}(\mathbb{X})=\mathbb{X}.
\]

The extensivity and monotonicity assertions in (iii) and (iv)
follow from Proposition~\ref{prop:Hstar-basic-operators}.

Finally, by
Corollary~\ref{cor:Hstar-closure-closed},
$H^{*}\operatorname{Cl}(\mathbb{A})$ is itself $H^{*}$-closed.
Applying
Corollary~\ref{cor:Hstar-fixed-points} therefore gives
\[
H^{*}\operatorname{Cl}
\bigl(H^{*}\operatorname{Cl}(\mathbb{A})\bigr)
=
H^{*}\operatorname{Cl}(\mathbb{A}).
\]
\end{proof}

\begin{theorem}
\label{thm:Hstar-interior-properties}
For all $\mathbb{A},\mathbb{B}\subseteq\mathbb{X}$, the
$H^{*}$-interior operator satisfies:
\begin{enumerate}[label=(\roman*)]

    \item
    \[
    H^{*}\operatorname{Int}(\varnothing)=\varnothing;
    \]

    \item
    \[
    H^{*}\operatorname{Int}(\mathbb{X})=\mathbb{X};
    \]

    \item
    \[
    H^{*}\operatorname{Int}(\mathbb{A})
    \subseteq\mathbb{A};
    \]

    \item if $\mathbb{A}\subseteq\mathbb{B}$, then
    \[
    H^{*}\operatorname{Int}(\mathbb{A})
    \subseteq
    H^{*}\operatorname{Int}(\mathbb{B});
    \]

    \item
    \[
    H^{*}\operatorname{Int}
    \bigl(H^{*}\operatorname{Int}(\mathbb{A})\bigr)
    =
    H^{*}\operatorname{Int}(\mathbb{A}).
    \]

\end{enumerate}
\end{theorem}

\begin{proof}
Since
\[
\varnothing,\mathbb{X}\in H^{*}O(\mathbb{X}),
\]
the first two assertions follow directly from the definition of
$H^{*}$-interior.

Assertions (iii) and (iv) follow from
Proposition~\ref{prop:Hstar-basic-operators}.

By Theorem~\ref{thm:Hstar-strong-generalized-topology},
$H^{*}O(\mathbb{X})$ is closed under arbitrary unions. Hence
$H^{*}\operatorname{Int}(\mathbb{A})$ is itself $H^{*}$-open.
Corollary~\ref{cor:Hstar-fixed-points} therefore yields
\[
H^{*}\operatorname{Int}
\bigl(H^{*}\operatorname{Int}(\mathbb{A})\bigr)
=
H^{*}\operatorname{Int}(\mathbb{A}).
\]
\end{proof}

The complement duality between the two operators has already been
established in Proposition~\ref{prop:Hstar-duality}; namely,
\begin{equation}
H^{*}\operatorname{Cl}(\mathbb{A})
=
\mathbb{X}\setminus
H^{*}\operatorname{Int}
(\mathbb{X}\setminus\mathbb{A}),
\label{eq:Hstar-closure-duality}
\end{equation}
and
\begin{equation}
H^{*}\operatorname{Int}(\mathbb{A})
=
\mathbb{X}\setminus
H^{*}\operatorname{Cl}
(\mathbb{X}\setminus\mathbb{A}).
\label{eq:Hstar-interior-duality}
\end{equation}

\subsection{Finite Unions and Intersections}

The following inclusions hold at the level of the
$H^{*}$-generated generalized topology without assuming that
$H^{*}O(\mathbb{X})$ is closed under finite intersections.

\begin{proposition}
\label{prop:Hstar-union-intersection}
For all $\mathbb{A},\mathbb{B}\subseteq\mathbb{X}$,
\begin{equation}
H^{*}\operatorname{Cl}(\mathbb{A})
\cup
H^{*}\operatorname{Cl}(\mathbb{B})
\subseteq
H^{*}\operatorname{Cl}(\mathbb{A}\cup\mathbb{B}),
\label{eq:Hstar-closure-union-inclusion}
\end{equation}
and
\begin{equation}
H^{*}\operatorname{Int}(\mathbb{A}\cap\mathbb{B})
\subseteq
H^{*}\operatorname{Int}(\mathbb{A})
\cap
H^{*}\operatorname{Int}(\mathbb{B}).
\label{eq:Hstar-interior-intersection-inclusion}
\end{equation}
\end{proposition}

\begin{proof}
Since
\[
\mathbb{A}\subseteq\mathbb{A}\cup\mathbb{B}
\quad\text{and}\quad
\mathbb{B}\subseteq\mathbb{A}\cup\mathbb{B},
\]
monotonicity of $H^{*}\operatorname{Cl}$ gives
\[
H^{*}\operatorname{Cl}(\mathbb{A})
\subseteq
H^{*}\operatorname{Cl}(\mathbb{A}\cup\mathbb{B})
\]
and
\[
H^{*}\operatorname{Cl}(\mathbb{B})
\subseteq
H^{*}\operatorname{Cl}(\mathbb{A}\cup\mathbb{B}).
\]
Taking the union of these inclusions yields
\eqref{eq:Hstar-closure-union-inclusion}.

Similarly,
\[
\mathbb{A}\cap\mathbb{B}\subseteq\mathbb{A}
\quad\text{and}\quad
\mathbb{A}\cap\mathbb{B}\subseteq\mathbb{B}.
\]
Monotonicity of $H^{*}\operatorname{Int}$ therefore gives
\[
H^{*}\operatorname{Int}(\mathbb{A}\cap\mathbb{B})
\subseteq
H^{*}\operatorname{Int}(\mathbb{A})
\]
and
\[
H^{*}\operatorname{Int}(\mathbb{A}\cap\mathbb{B})
\subseteq
H^{*}\operatorname{Int}(\mathbb{B}).
\]
Hence
\[
H^{*}\operatorname{Int}(\mathbb{A}\cap\mathbb{B})
\subseteq
H^{*}\operatorname{Int}(\mathbb{A})
\cap
H^{*}\operatorname{Int}(\mathbb{B}).
\]
\end{proof}

\begin{remark}
\label{rem:Hstar-no-equality-general}
At the level of a strong generalized topology, the inclusions in
Proposition~\ref{prop:Hstar-union-intersection} need not become
equalities. The reverse inclusions require additional finite
intersection or, equivalently, finite union structure. Thus these
equalities arise naturally when the associated ordinary topology
$\tau_{H^{*}}$ exists.
\end{remark}

\subsection{The Ordinary Topological Case}

We now impose the additional condition that the $H^{*}$-generated
generalized topology is an ordinary topology.

\begin{theorem}[Finite Union and Intersection Laws]
\label{thm:Hstar-Kuratowski}
Suppose that the equivalent conditions of
Theorem~\ref{thm:Hstar-topology-characterization} hold. Then, for all
$\mathbb{A},\mathbb{B}\subseteq\mathbb{X}$,
\begin{equation}
H^{*}\operatorname{Cl}(\mathbb{A}\cup\mathbb{B})
=
H^{*}\operatorname{Cl}(\mathbb{A})
\cup
H^{*}\operatorname{Cl}(\mathbb{B}),
\label{eq:Hstar-closure-finite-union}
\end{equation}
and
\begin{equation}
H^{*}\operatorname{Int}(\mathbb{A}\cap\mathbb{B})
=
H^{*}\operatorname{Int}(\mathbb{A})
\cap
H^{*}\operatorname{Int}(\mathbb{B}).
\label{eq:Hstar-interior-finite-intersection}
\end{equation}
\end{theorem}

\begin{proof}
Under the stated hypothesis,
\[
\tau_{H^{*}}=H^{*}O(\mathbb{X})
\]
is an ordinary topology. By
Theorem~\ref{thm:Hstar-topology-characterization},
$H^{*}C(\mathbb{X})$ is therefore closed under finite unions.

Since
$H^{*}\operatorname{Cl}(\mathbb{A})$ and
$H^{*}\operatorname{Cl}(\mathbb{B})$ are $H^{*}$-closed, their
union is $H^{*}$-closed and contains
$\mathbb{A}\cup\mathbb{B}$. By the definition of
$H^{*}$-closure,
\[
H^{*}\operatorname{Cl}(\mathbb{A}\cup\mathbb{B})
\subseteq
H^{*}\operatorname{Cl}(\mathbb{A})
\cup
H^{*}\operatorname{Cl}(\mathbb{B}).
\]
The reverse inclusion follows from
Proposition~\ref{prop:Hstar-union-intersection}. Hence
\[
H^{*}\operatorname{Cl}(\mathbb{A}\cup\mathbb{B})
=
H^{*}\operatorname{Cl}(\mathbb{A})
\cup
H^{*}\operatorname{Cl}(\mathbb{B}).
\]

The corresponding interior identity follows either by the dual
argument or from
\eqref{eq:Hstar-closure-duality} and De Morgan's law.
\end{proof}

The preceding theorem shows that, once the associated topology
exists, the $H^{*}$-closure and $H^{*}$-interior satisfy the
corresponding finite-set laws of ordinary topological closure and
interior. The following result makes this relationship precise.

\begin{corollary}
\label{cor:Hstar-associated-closure}
If
\[
\tau_{H^{*}}=H^{*}O(\mathbb{X})
\]
is an ordinary topology, then, for every
$\mathbb{A}\subseteq\mathbb{X}$,
\[
H^{*}\operatorname{Cl}(\mathbb{A})
=
\operatorname{Cl}_{\tau_{H^{*}}}(\mathbb{A})
\]
and
\[
H^{*}\operatorname{Int}(\mathbb{A})
=
\operatorname{Int}_{\tau_{H^{*}}}(\mathbb{A}).
\]
\end{corollary}

\begin{proof}
The closed sets of $(\mathbb{X},\tau_{H^{*}})$ are precisely the
complements of members of
\[
\tau_{H^{*}}=H^{*}O(\mathbb{X}),
\]
and therefore they are exactly the members of
$H^{*}C(\mathbb{X})$. Hence the ordinary
$\tau_{H^{*}}$-closure and the $H^{*}$-closure are obtained from
the same family of closed supersets. Consequently,
\[
H^{*}\operatorname{Cl}(\mathbb{A})
=
\operatorname{Cl}_{\tau_{H^{*}}}(\mathbb{A}).
\]

Similarly, the open sets of $\tau_{H^{*}}$ are exactly the members
of $H^{*}O(\mathbb{X})$. Thus
\[
H^{*}\operatorname{Int}(\mathbb{A})
=
\operatorname{Int}_{\tau_{H^{*}}}(\mathbb{A}).
\]
\end{proof}

\section{Separation Properties of the $H^{*}$-Generated Structure}
\label{sec:separation}

We now examine separation properties determined by $H^{*}$-open
sets. Throughout this section, we assume that
$H^{*}C(\mathbb{X})$ is closed under arbitrary intersections. Hence,
by Theorem~\ref{thm:Hstar-strong-generalized-topology},
\[
\Gamma_{H^{*}}=H^{*}O(\mathbb{X})
\]
is a strong generalized topology on $\mathbb{X}$.

The $H^{*}$-separation axioms considered below are recalled from the
existing $H^{*}$-separation framework; see, for example,
\cite{tomarsome2024}. Their role here is to relate separation in the
$H^{*}$-generated structure to separation in the associated topology,
whenever that topology exists.

\begin{definition}
\label{def:Hstar-separation}
Let $(\mathbb{X},\tau)$ be a topological space.
\begin{enumerate}[label=(\roman*)]

    \item The space $(\mathbb{X},\tau)$ is said to be
    \emph{$H^{*}$-$T_{0}$} if, for every pair of distinct points
    $k,l\in\mathbb{X}$, there exists an $H^{*}$-open set containing
    one of them but not the other.

    \item The space $(\mathbb{X},\tau)$ is said to be
    \emph{$H^{*}$-$T_{1}$} if, for every pair of distinct points
    $k,l\in\mathbb{X}$, there exist $H^{*}$-open sets
    $\mathbb{U}$ and $\mathbb{V}$ such that
    \[
    k\in\mathbb{U},\qquad l\notin\mathbb{U},
    \]
    and
    \[
    l\in\mathbb{V},\qquad k\notin\mathbb{V}.
    \]

    \item The space $(\mathbb{X},\tau)$ is said to be
    \emph{$H^{*}$-$T_{2}$} if, for every pair of distinct points
    $k,l\in\mathbb{X}$, there exist disjoint $H^{*}$-open sets
    $\mathbb{U}$ and $\mathbb{V}$ such that
    \[
    k\in\mathbb{U},
    \qquad
    l\in\mathbb{V}.
    \]

\end{enumerate}
\end{definition}

The definitions immediately yield the standard implication chain.

\begin{proposition}
\label{prop:Hstar-separation-chain}
For a topological space $(\mathbb{X},\tau)$,
\[
H^{*}\text{-}T_{2}
\Longrightarrow
H^{*}\text{-}T_{1}
\Longrightarrow
H^{*}\text{-}T_{0}.
\]
\end{proposition}

\begin{proof}
Suppose that $(\mathbb{X},\tau)$ is $H^{*}$-$T_{2}$ and let
$k,l\in\mathbb{X}$ with $k\neq l$. Then there exist disjoint
$H^{*}$-open sets $\mathbb{U}$ and $\mathbb{V}$ such that
\[
k\in\mathbb{U}
\qquad\text{and}\qquad
l\in\mathbb{V}.
\]
Since
\[
\mathbb{U}\cap\mathbb{V}=\varnothing,
\]
we have
\[
l\notin\mathbb{U}
\qquad\text{and}\qquad
k\notin\mathbb{V}.
\]
Hence $(\mathbb{X},\tau)$ is $H^{*}$-$T_{1}$.

The implication
\[
H^{*}\text{-}T_{1}
\Longrightarrow
H^{*}\text{-}T_{0}
\]
follows immediately from Definition~\ref{def:Hstar-separation}.
\end{proof}

Under the standing generalized-topological hypothesis, the
$H^{*}$-$T_{1}$ condition admits the following useful
characterization.

\begin{theorem}
\label{thm:Hstar-T1-singletons}
The space $(\mathbb{X},\tau)$ is $H^{*}$-$T_{1}$ if and only if
every singleton subset of $\mathbb{X}$ is $H^{*}$-closed.
\end{theorem}

\begin{proof}
Suppose first that $(\mathbb{X},\tau)$ is $H^{*}$-$T_{1}$.
Fix $k\in\mathbb{X}$. For every
$l\in\mathbb{X}\setminus\{k\}$, there exists an $H^{*}$-open set
$\mathbb{U}_{l}$ such that
\[
l\in\mathbb{U}_{l}
\qquad\text{and}\qquad
k\notin\mathbb{U}_{l}.
\]
Therefore,
\[
\mathbb{X}\setminus\{k\}
=
\bigcup_{l\in\mathbb{X}\setminus\{k\}}
\mathbb{U}_{l}.
\]
Since $H^{*}O(\mathbb{X})$ is closed under arbitrary unions,
$\mathbb{X}\setminus\{k\}$ is $H^{*}$-open. Hence
$\{k\}$ is $H^{*}$-closed.

Conversely, suppose that every singleton subset of $\mathbb{X}$ is
$H^{*}$-closed. Let $k,l\in\mathbb{X}$ with $k\neq l$. Then
\[
\mathbb{X}\setminus\{l\}
\quad\text{and}\quad
\mathbb{X}\setminus\{k\}
\]
are $H^{*}$-open. Moreover,
\[
k\in\mathbb{X}\setminus\{l\},
\qquad
l\notin\mathbb{X}\setminus\{l\},
\]
and
\[
l\in\mathbb{X}\setminus\{k\},
\qquad
k\notin\mathbb{X}\setminus\{k\}.
\]
Thus $(\mathbb{X},\tau)$ is $H^{*}$-$T_{1}$.
\end{proof}

\subsection{$H^{*}$-Difference Separation}
\label{subsec:Hstar-D-separation}

We next recall the $H^{*}$-difference-set separation notions used in
our earlier study of $H^{*}$-difference spaces. They are included
here in order to place them within the generalized-topological and
associated-topological framework developed in the present paper.

\begin{definition}
\label{def:Hstar-D-set}
A subset $\mathbb{D}\subseteq\mathbb{X}$ is called an
\emph{$H^{*}$-difference set}, briefly an \emph{$H^{*}D$-set}, if
there exist
\[
\mathbb{U},\mathbb{V}\in H^{*}O(\mathbb{X})
\]
such that
\[
\mathbb{D}=\mathbb{U}\setminus\mathbb{V}
\]
and
\[
\mathbb{U}\neq\mathbb{X}.
\]
\end{definition}

\begin{definition}
\label{def:Hstar-D1-D2}
The space $(\mathbb{X},\tau)$ is said to be:
\begin{enumerate}[label=(\roman*)]

    \item \emph{$H^{*}$-$D_{1}$} if, for every pair of distinct
    points $k,l\in\mathbb{X}$, there exist $H^{*}D$-sets
    $\mathbb{D}_{1}$ and $\mathbb{D}_{2}$ such that
    \[
    k\in\mathbb{D}_{1},\qquad
    l\notin\mathbb{D}_{1},
    \]
    and
    \[
    l\in\mathbb{D}_{2},\qquad
    k\notin\mathbb{D}_{2};
    \]

    \item \emph{$H^{*}$-$D_{2}$} if, for every pair of distinct
    points $k,l\in\mathbb{X}$, there exist disjoint $H^{*}D$-sets
    $\mathbb{D}_{1}$ and $\mathbb{D}_{2}$ such that
    \[
    k\in\mathbb{D}_{1},
    \qquad
    l\in\mathbb{D}_{2}.
    \]

\end{enumerate}
\end{definition}

For completeness, we recall the following elementary implications.

\begin{proposition}
\label{prop:Hstar-D2-D1}
Every $H^{*}$-$D_{2}$ space is an $H^{*}$-$D_{1}$ space.
\end{proposition}

\begin{proof}
Let $k,l\in\mathbb{X}$ be distinct. Since the space is
$H^{*}$-$D_{2}$, there exist disjoint $H^{*}D$-sets
$\mathbb{D}_{1}$ and $\mathbb{D}_{2}$ such that
\[
k\in\mathbb{D}_{1}
\qquad\text{and}\qquad
l\in\mathbb{D}_{2}.
\]
Their disjointness implies
\[
l\notin\mathbb{D}_{1}
\qquad\text{and}\qquad
k\notin\mathbb{D}_{2}.
\]
Hence the defining conditions of an $H^{*}$-$D_{1}$ space are
satisfied.
\end{proof}

\begin{proposition}
\label{prop:Hstar-D1-T0}
Every $H^{*}$-$D_{1}$ space is an $H^{*}$-$T_{0}$ space.
\end{proposition}

\begin{proof}
Let $k,l\in\mathbb{X}$ with $k\neq l$. Since
$(\mathbb{X},\tau)$ is $H^{*}$-$D_{1}$, there exists an
$H^{*}D$-set $\mathbb{D}$ such that
\[
k\in\mathbb{D}
\qquad\text{and}\qquad
l\notin\mathbb{D}.
\]
Write
\[
\mathbb{D}=\mathbb{U}\setminus\mathbb{V},
\]
where
\[
\mathbb{U},\mathbb{V}\in H^{*}O(\mathbb{X})
\qquad\text{and}\qquad
\mathbb{U}\neq\mathbb{X}.
\]
Since $k\in\mathbb{D}$,
\[
k\in\mathbb{U}
\qquad\text{and}\qquad
k\notin\mathbb{V}.
\]

If $l\notin\mathbb{U}$, then $\mathbb{U}$ is an $H^{*}$-open set
containing $k$ but not $l$.

If $l\in\mathbb{U}$, then $l\notin\mathbb{D}$ implies
$l\in\mathbb{V}$. Hence $\mathbb{V}$ is an $H^{*}$-open set
containing $l$ but not $k$.

Thus one of the two points belongs to an $H^{*}$-open set that does
not contain the other. Therefore $(\mathbb{X},\tau)$ is
$H^{*}$-$T_{0}$.
\end{proof}

\subsection{Separation Axioms in the Associated $H^{*}$-Topology}
\label{subsec:associated-separation}

We now interpret the preceding $H^{*}$-separation properties through
the associated topology. Throughout this subsection, we assume that
the equivalent conditions of
Theorem~\ref{thm:Hstar-topology-characterization} are satisfied.
Thus,
\[
\tau_{H^{*}}=H^{*}O(\mathbb{X})
\]
is an ordinary topology on $\mathbb{X}$.

Since the open sets of $\tau_{H^{*}}$ are precisely the
$H^{*}$-open subsets of $\mathbb{X}$, the corresponding separation
conditions agree.

\begin{theorem}
\label{thm:Hstar-T0-associated}
The space $(\mathbb{X},\tau)$ is $H^{*}$-$T_{0}$ if and only if
$(\mathbb{X},\tau_{H^{*}})$ is a $T_{0}$ space.
\end{theorem}

\begin{proof}
The assertion follows immediately from
\[
\tau_{H^{*}}=H^{*}O(\mathbb{X}),
\]
since both conditions require exactly the same open-set separation
of every pair of distinct points.
\end{proof}

\begin{theorem}
\label{thm:Hstar-T1-associated}
The space $(\mathbb{X},\tau)$ is $H^{*}$-$T_{1}$ if and only if
$(\mathbb{X},\tau_{H^{*}})$ is a $T_{1}$ space.
\end{theorem}

\begin{proof}
Since
\[
\tau_{H^{*}}=H^{*}O(\mathbb{X}),
\]
the $H^{*}$-open sets occurring in
Definition~\ref{def:Hstar-separation} are precisely the open sets of
$\tau_{H^{*}}$. Hence the two $T_{1}$ conditions are identical.
\end{proof}

\begin{theorem}
\label{thm:Hstar-T2-associated}
The space $(\mathbb{X},\tau)$ is $H^{*}$-$T_{2}$ if and only if
$(\mathbb{X},\tau_{H^{*}})$ is Hausdorff.
\end{theorem}

\begin{proof}
Again,
\[
\tau_{H^{*}}=H^{*}O(\mathbb{X}).
\]
Thus disjoint $H^{*}$-open neighbourhoods of two distinct points are
exactly disjoint $\tau_{H^{*}}$-open neighbourhoods of those points.
The equivalence follows.
\end{proof}

\begin{corollary}
\label{cor:Hstar-Ti-associated}
For $i\in\{0,1,2\}$,
\[
(\mathbb{X},\tau)
\text{ is }H^{*}\text{-}T_{i}
\quad\Longleftrightarrow\quad
(\mathbb{X},\tau_{H^{*}})
\text{ is }T_{i}.
\]
\end{corollary}

\subsection{$H^{*}$-Difference Separation and the Associated Topology}
\label{subsec:Hstar-D-associated}

For comparison, recall the classical difference-set notion.

\begin{definition}
\label{def:classical-D-set}
Let $(\mathbb{X},\sigma)$ be a topological space. A subset
$\mathbb{D}\subseteq\mathbb{X}$ is called a \emph{$D$-set} if
\[
\mathbb{D}=\mathbb{U}\setminus\mathbb{V},
\]
where
\[
\mathbb{U},\mathbb{V}\in\sigma
\qquad\text{and}\qquad
\mathbb{U}\neq\mathbb{X}.
\]
\end{definition}

\begin{proposition}
\label{prop:Hstar-D-classical-D}
A subset $\mathbb{D}\subseteq\mathbb{X}$ is an $H^{*}D$-set in
$(\mathbb{X},\tau)$ if and only if it is a $D$-set in
$(\mathbb{X},\tau_{H^{*}})$.
\end{proposition}

\begin{proof}
By Definition~\ref{def:Hstar-D-set}, $\mathbb{D}$ is an
$H^{*}D$-set if and only if
\[
\mathbb{D}=\mathbb{U}\setminus\mathbb{V},
\]
where
\[
\mathbb{U},\mathbb{V}\in H^{*}O(\mathbb{X})
\qquad\text{and}\qquad
\mathbb{U}\neq\mathbb{X}.
\]
Since
\[
H^{*}O(\mathbb{X})=\tau_{H^{*}},
\]
this is precisely the definition of a $D$-set in
$(\mathbb{X},\tau_{H^{*}})$.
\end{proof}

\begin{theorem}
\label{thm:Hstar-D-associated}
For $i\in\{1,2\}$,
\[
(\mathbb{X},\tau)
\text{ is }H^{*}\text{-}D_{i}
\quad\Longleftrightarrow\quad
(\mathbb{X},\tau_{H^{*}})
\text{ is }D_{i}.
\]
\end{theorem}

\begin{proof}
By Proposition~\ref{prop:Hstar-D-classical-D}, the $H^{*}D$-sets of
$(\mathbb{X},\tau)$ are exactly the $D$-sets of
$(\mathbb{X},\tau_{H^{*}})$. Hence the point-separation conditions
defining $H^{*}$-$D_{1}$ and $H^{*}$-$D_{2}$ coincide,
respectively, with the classical $D_{1}$ and $D_{2}$ conditions in
the associated topology.
\end{proof}

\begin{corollary}
\label{cor:Hstar-D1-D2-associated}
If the classical $D_{1}$ and $D_{2}$ conditions are equivalent in
$(\mathbb{X},\tau_{H^{*}})$, then
\[
(\mathbb{X},\tau)
\text{ is }H^{*}\text{-}D_{1}
\quad\Longleftrightarrow\quad
(\mathbb{X},\tau)
\text{ is }H^{*}\text{-}D_{2}.
\]
\end{corollary}

\begin{proof}
This follows immediately from
Theorem~\ref{thm:Hstar-D-associated}.
\end{proof}

\begin{remark}
\label{rem:associated-separation-interpretation}
Whenever the associated topology exists, the preceding equivalences
should be viewed as structural identifications rather than as
independent separation phenomena. Indeed, the open sets of
$(\mathbb{X},\tau_{H^{*}})$ are exactly the $H^{*}$-open sets of
$(\mathbb{X},\tau)$. Consequently, the $H^{*}$-$T_i$ conditions
and the corresponding $H^{*}$-difference separation conditions
translate directly into their classical counterparts in the
associated topology. This provides the bridge needed for the
higher separation properties considered in the next section.
\end{remark}

\section{Higher Separation Properties in the Associated
$H^{*}$-Topology}
\label{sec:higher-separation}

The preceding section establishes that, whenever the associated
topology exists, the point-separation properties defined in terms of
$H^{*}$-open sets correspond to the usual separation properties of
$(\mathbb{X},\tau_{H^{*}})$. We now extend this interpretation to
separation conditions involving points and closed sets, as well as
pairs of disjoint closed sets.

Throughout this section, we assume that the equivalent conditions of
Theorem~\ref{thm:Hstar-topology-characterization} are satisfied, so
that
\[
\tau_{H^{*}}=H^{*}O(\mathbb{X})
\]
is an ordinary topology on $\mathbb{X}$. Consequently, the closed
sets of $(\mathbb{X},\tau_{H^{*}})$ are precisely the members of
$H^{*}C(\mathbb{X})$.

\subsection{$H^{*}$-Regularity}
\label{subsec:Hstar-regularity}

\begin{definition}
\label{def:Hstar-regular}
A topological space $(\mathbb{X},\tau)$ is said to be
\emph{$H^{*}$-regular} if, whenever
$\mathbb{F}\in H^{*}C(\mathbb{X})$ and
$k\in\mathbb{X}\setminus\mathbb{F}$, there exist disjoint
$H^{*}$-open sets $\mathbb{U}$ and $\mathbb{V}$ such that
\[
k\in\mathbb{U}
\qquad\text{and}\qquad
\mathbb{F}\subseteq\mathbb{V}.
\]
\end{definition}

The following result identifies this condition with ordinary
regularity in the associated topology.

\begin{theorem}
\label{thm:Hstar-regular-associated}
The space $(\mathbb{X},\tau)$ is $H^{*}$-regular if and only if
$(\mathbb{X},\tau_{H^{*}})$ is regular.
\end{theorem}

\begin{proof}
Suppose that $(\mathbb{X},\tau)$ is $H^{*}$-regular. Let
$\mathbb{F}$ be a closed subset of
$(\mathbb{X},\tau_{H^{*}})$ and let
\[
k\in\mathbb{X}\setminus\mathbb{F}.
\]
Since the $\tau_{H^{*}}$-closed sets are precisely the
$H^{*}$-closed sets,
\[
\mathbb{F}\in H^{*}C(\mathbb{X}).
\]
By $H^{*}$-regularity, there exist disjoint $H^{*}$-open sets
$\mathbb{U}$ and $\mathbb{V}$ such that
\[
k\in\mathbb{U},
\qquad
\mathbb{F}\subseteq\mathbb{V}.
\]
Since
\[
H^{*}O(\mathbb{X})=\tau_{H^{*}},
\]
the sets $\mathbb{U}$ and $\mathbb{V}$ are
$\tau_{H^{*}}$-open. Hence $(\mathbb{X},\tau_{H^{*}})$ is regular.

Conversely, suppose that $(\mathbb{X},\tau_{H^{*}})$ is regular.
Let $\mathbb{F}\in H^{*}C(\mathbb{X})$ and let
\[
k\in\mathbb{X}\setminus\mathbb{F}.
\]
Then $\mathbb{F}$ is $\tau_{H^{*}}$-closed. By regularity, there
exist disjoint $\tau_{H^{*}}$-open sets $\mathbb{U}$ and
$\mathbb{V}$ such that
\[
k\in\mathbb{U},
\qquad
\mathbb{F}\subseteq\mathbb{V}.
\]
Since the $\tau_{H^{*}}$-open sets are exactly the $H^{*}$-open
sets, $(\mathbb{X},\tau)$ is $H^{*}$-regular.
\end{proof}

\begin{definition}
\label{def:Hstar-T3}
The space $(\mathbb{X},\tau)$ is called an
\emph{$H^{*}$-$T_{3}$ space} if it is both $H^{*}$-regular and
$H^{*}$-$T_{1}$.
\end{definition}

\begin{corollary}
\label{cor:Hstar-T3-associated}
The space $(\mathbb{X},\tau)$ is $H^{*}$-$T_{3}$ if and only if
$(\mathbb{X},\tau_{H^{*}})$ is $T_{3}$.
\end{corollary}

\begin{proof}
This follows from
Theorem~\ref{thm:Hstar-regular-associated} together with
Theorem~\ref{thm:Hstar-T1-associated}.
\end{proof}

\subsection{$H^{*}$-Normality}
\label{subsec:Hstar-normality}

The notion of $H^{*}$-normality has been considered previously in
the literature; see Sharma, Sharma and Tomar
\cite{sharma2015h}. We recall it here in order to identify its
relation with normality in the associated topology.

\begin{definition}
\label{def:Hstar-normal}
A topological space $(\mathbb{X},\tau)$ is said to be
\emph{$H^{*}$-normal} if, for every pair of disjoint
$H^{*}$-closed sets $\mathbb{F}_{1}$ and $\mathbb{F}_{2}$, there
exist disjoint $H^{*}$-open sets $\mathbb{U}_{1}$ and
$\mathbb{U}_{2}$ such that
\[
\mathbb{F}_{1}\subseteq\mathbb{U}_{1}
\qquad\text{and}\qquad
\mathbb{F}_{2}\subseteq\mathbb{U}_{2}.
\]
\end{definition}

The associated topology gives the following direct interpretation of
$H^{*}$-normality.

\begin{theorem}
\label{thm:Hstar-normal-associated}
The space $(\mathbb{X},\tau)$ is $H^{*}$-normal if and only if
$(\mathbb{X},\tau_{H^{*}})$ is normal.
\end{theorem}

\begin{proof}
Suppose first that $(\mathbb{X},\tau)$ is $H^{*}$-normal. Let
$\mathbb{F}_{1}$ and $\mathbb{F}_{2}$ be disjoint closed subsets of
$(\mathbb{X},\tau_{H^{*}})$. Since the closed sets of
$\tau_{H^{*}}$ are exactly the $H^{*}$-closed sets,
\[
\mathbb{F}_{1},\mathbb{F}_{2}\in H^{*}C(\mathbb{X}).
\]
By $H^{*}$-normality, there exist disjoint $H^{*}$-open sets
$\mathbb{U}_{1}$ and $\mathbb{U}_{2}$ such that
\[
\mathbb{F}_{1}\subseteq\mathbb{U}_{1},
\qquad
\mathbb{F}_{2}\subseteq\mathbb{U}_{2}.
\]
Since
\[
H^{*}O(\mathbb{X})=\tau_{H^{*}},
\]
the sets $\mathbb{U}_{1}$ and $\mathbb{U}_{2}$ are
$\tau_{H^{*}}$-open. Therefore
$(\mathbb{X},\tau_{H^{*}})$ is normal.

Conversely, suppose that $(\mathbb{X},\tau_{H^{*}})$ is normal.
Let
\[
\mathbb{F}_{1},\mathbb{F}_{2}\in H^{*}C(\mathbb{X})
\]
be disjoint. Then $\mathbb{F}_{1}$ and $\mathbb{F}_{2}$ are closed
in $\tau_{H^{*}}$. Hence there exist disjoint
$\tau_{H^{*}}$-open sets $\mathbb{U}_{1}$ and $\mathbb{U}_{2}$
such that
\[
\mathbb{F}_{1}\subseteq\mathbb{U}_{1},
\qquad
\mathbb{F}_{2}\subseteq\mathbb{U}_{2}.
\]
These sets are $H^{*}$-open. Thus
$(\mathbb{X},\tau)$ is $H^{*}$-normal.
\end{proof}

\begin{definition}
\label{def:Hstar-T4}
The space $(\mathbb{X},\tau)$ is called an
\emph{$H^{*}$-$T_{4}$ space} if it is both $H^{*}$-normal and
$H^{*}$-$T_{1}$.
\end{definition}

\begin{corollary}
\label{cor:Hstar-T4-associated}
The space $(\mathbb{X},\tau)$ is $H^{*}$-$T_{4}$ if and only if
$(\mathbb{X},\tau_{H^{*}})$ is $T_{4}$.
\end{corollary}

\begin{proof}
This follows from
Theorem~\ref{thm:Hstar-normal-associated} and
Theorem~\ref{thm:Hstar-T1-associated}.
\end{proof}

\subsection{Complete Regularity and the $H^{*}$-Framework}
\label{subsec:Hstar-complete-regularity}

We next formulate complete regularity relative to the associated
$H^{*}$-topology.

\begin{definition}
\label{def:Hstar-completely-regular}
A topological space $(\mathbb{X},\tau)$ is said to be
\emph{$H^{*}$-completely regular} if, whenever
$\mathbb{F}\in H^{*}C(\mathbb{X})$ and
$k\in\mathbb{X}\setminus\mathbb{F}$, there exists a function
\[
f:\mathbb{X}\longrightarrow[0,1]
\]
which is continuous from
$(\mathbb{X},\tau_{H^{*}})$ to $[0,1]$ endowed with its usual
topology and satisfies
\[
f(k)=0
\qquad\text{and}\qquad
f(\mathbb{F})=\{1\}.
\]
\end{definition}

\begin{theorem}
\label{thm:Hstar-completely-regular-associated}
The space $(\mathbb{X},\tau)$ is $H^{*}$-completely regular if and
only if $(\mathbb{X},\tau_{H^{*}})$ is completely regular.
\end{theorem}

\begin{proof}
The $\tau_{H^{*}}$-closed subsets of $\mathbb{X}$ are precisely the
members of $H^{*}C(\mathbb{X})$. Moreover, the continuity required
in Definition~\ref{def:Hstar-completely-regular} is exactly
continuity with respect to $\tau_{H^{*}}$. Therefore the defining
point--closed-set separation condition for $H^{*}$-complete
regularity is identical to the classical complete regularity
condition in $(\mathbb{X},\tau_{H^{*}})$.
\end{proof}

\begin{definition}
\label{def:Hstar-Tychonoff}
An $H^{*}$-completely regular $H^{*}$-$T_{1}$ space is called an
\emph{$H^{*}$-Tychonoff space}.
\end{definition}

\begin{corollary}
\label{cor:Hstar-Tychonoff-associated}
The space $(\mathbb{X},\tau)$ is $H^{*}$-Tychonoff if and only if
$(\mathbb{X},\tau_{H^{*}})$ is Tychonoff.
\end{corollary}

\begin{proof}
This follows from
Theorem~\ref{thm:Hstar-completely-regular-associated} and
Theorem~\ref{thm:Hstar-T1-associated}.
\end{proof}

\subsection{A Transfer Principle}
\label{subsec:Hstar-transfer-principle}

The preceding results are instances of a common structural
observation. We formulate this observation carefully in order to
clarify the scope of the associated-topology method.

\begin{theorem}[Transfer Principle]
\label{thm:Hstar-transfer-principle}
Suppose that
\[
\tau_{H^{*}}=H^{*}O(\mathbb{X})
\]
is an ordinary topology on $\mathbb{X}$. Then every property
formulated exclusively in terms of the members of
$H^{*}O(\mathbb{X})$, their complements, and the corresponding
closure and interior operators admits the equivalent classical
formulation in $(\mathbb{X},\tau_{H^{*}})$.
\end{theorem}

\begin{proof}
By Definition~\ref{def:associated-Hstar-topology},
\[
\tau_{H^{*}}=H^{*}O(\mathbb{X}).
\]
Thus the open subsets of $(\mathbb{X},\tau_{H^{*}})$ are precisely
the $H^{*}$-open subsets of $\mathbb{X}$. Taking complements shows
that the $\tau_{H^{*}}$-closed sets are precisely the members of
$H^{*}C(\mathbb{X})$.

Furthermore, by
Corollary~\ref{cor:Hstar-associated-closure}, for every
$\mathbb{A}\subseteq\mathbb{X}$,
\[
H^{*}\operatorname{Cl}(\mathbb{A})
=
\operatorname{Cl}_{\tau_{H^{*}}}(\mathbb{A})
\]
and
\[
H^{*}\operatorname{Int}(\mathbb{A})
=
\operatorname{Int}_{\tau_{H^{*}}}(\mathbb{A}).
\]
Consequently, a property expressed solely through these objects has
the same formulation after replacing the $H^{*}$-open,
$H^{*}$-closed, $H^{*}$-closure, and $H^{*}$-interior structures
by their corresponding objects in
$(\mathbb{X},\tau_{H^{*}})$.
\end{proof}

\begin{remark}
\label{rem:transfer-limitation}
Theorem~\ref{thm:Hstar-transfer-principle} is an identification
principle and should not be interpreted as asserting that
$(\mathbb{X},\tau_{H^{*}})$ automatically satisfies any particular
classical separation axiom. For example, regularity, complete
regularity, normality, or perfect normality must still be established
from appropriate additional hypotheses.

Likewise, stronger separation conclusions proved for other
generalized-topological constructions cannot be transferred to the
present $H^{*}$-framework merely because
$H^{*}O(\mathbb{X})$ forms a topology. Such conclusions require
their own hypotheses and proofs.
\end{remark}

\section{Examples and Structural Consequences}
\label{sec:examples}

The preceding results show that the structure determined by
$H^{*}$-open sets may differ substantially from the original
topology. In particular, even when the original space has weak
separation properties, the associated topology
$\tau_{H^{*}}$, whenever it exists, may be considerably finer.

The following examples, which arise naturally in the study of
$H^{*}$-open and $H^{*}$-difference structures, are revisited here
to illustrate this phenomenon in the present generalized-topological
framework.

\subsection{A Finite Example}
\label{subsec:finite-example}

\begin{example}
\label{ex:finite-discrete-Hstar}
Let
\[
\mathbb{X}=\{k,l,m\}
\]
and consider the topology
\[
\tau=
\left\{
\varnothing,\mathbb{X},\{k,l\}
\right\}.
\]
For this space, every subset of $\mathbb{X}$ is $h$-open as well
as $h$-closed. Consequently,
\[
h\text{-}\operatorname{Cl}(\mathbb{A})=\mathbb{A}
\]
for every $\mathbb{A}\subseteq\mathbb{X}$.

It follows from the definition of an $H^{*}$-closed set that every
subset of $\mathbb{X}$ is $H^{*}$-closed. Hence
\[
H^{*}C(\mathbb{X})=\mathcal{P}(\mathbb{X}),
\]
and, by taking complements,
\[
H^{*}O(\mathbb{X})=\mathcal{P}(\mathbb{X}).
\]
Therefore,
\[
\tau_{H^{*}}
=
H^{*}O(\mathbb{X})
=
\mathcal{P}(\mathbb{X}),
\]
which is the discrete topology on $\mathbb{X}$.
\end{example}

\begin{proposition}
\label{prop:finite-example-separation}
For the space in Example~\ref{ex:finite-discrete-Hstar},
$(\mathbb{X},\tau_{H^{*}})$ is $T_{0}$, $T_{1}$, Hausdorff,
regular, completely regular, and normal.
\end{proposition}

\begin{proof}
Since
\[
\tau_{H^{*}}=\mathcal{P}(\mathbb{X}),
\]
the associated space is discrete. Hence every subset of
$\mathbb{X}$ is both open and closed. The asserted separation
properties follow from the corresponding standard properties of a
discrete topological space.
\end{proof}

The original topology has a different separation behaviour.

\begin{proposition}
\label{prop:finite-original-not-T0}
The original space $(\mathbb{X},\tau)$ of
Example~\ref{ex:finite-discrete-Hstar} is not $T_{0}$.
\end{proposition}

\begin{proof}
Consider the distinct points $k,l\in\mathbb{X}$. The only nonempty
proper open set is
\[
\{k,l\},
\]
which contains both $k$ and $l$. Thus there is no open set in
$\tau$ containing exactly one of these two points. Therefore
$(\mathbb{X},\tau)$ is not $T_{0}$.
\end{proof}

\begin{remark}
\label{rem:finite-separation-improvement}
Example~\ref{ex:finite-discrete-Hstar} shows that the separation
properties of the associated $H^{*}$-topology need not reflect
those of the original topology. Indeed,
\[
(\mathbb{X},\tau)\text{ is not }T_{0},
\]
whereas
\[
(\mathbb{X},\tau_{H^{*}})
\]
is discrete.

In view of
Corollary~\ref{cor:Hstar-Ti-associated}, the original
$H^{*}$-structure is therefore $H^{*}$-$T_{0}$,
$H^{*}$-$T_{1}$, and $H^{*}$-$T_{2}$, even though the original
topology $\tau$ is not $T_{0}$.
\end{remark}

\subsection{An Infinite Example}
\label{subsec:infinite-example}

\begin{example}
\label{ex:infinite-discrete-Hstar}
Let $\mathbb{A}$ be a set with
\[
|\mathbb{A}|\geq 2,
\]
let $p\notin\mathbb{A}$, and put
\[
\mathbb{X}=\mathbb{A}\cup\{p\}.
\]
Consider the topology
\[
\tau=
\left\{
\varnothing,\mathbb{A},\mathbb{X}
\right\}.
\]

For this space, every subset of $\mathbb{X}$ is $h$-open and
$h$-closed. Consequently,
\[
h\text{-}\operatorname{Cl}(\mathbb{B})=\mathbb{B}
\]
for every $\mathbb{B}\subseteq\mathbb{X}$.

It follows from the defining condition of $H^{*}$-closed sets that
every subset of $\mathbb{X}$ is $H^{*}$-closed. Therefore,
\[
H^{*}C(\mathbb{X})
=
\mathcal{P}(\mathbb{X})
\]
and hence
\[
H^{*}O(\mathbb{X})
=
\mathcal{P}(\mathbb{X}).
\]
Thus
\[
\tau_{H^{*}}
=
\mathcal{P}(\mathbb{X}),
\]
so the associated $H^{*}$-topology is discrete.
\end{example}

\begin{proposition}
\label{prop:infinite-original-not-T0}
The original space $(\mathbb{X},\tau)$ in
Example~\ref{ex:infinite-discrete-Hstar} is not $T_{0}$.
\end{proposition}

\begin{proof}
Since $|\mathbb{A}|\geq2$, choose distinct points
$k,l\in\mathbb{A}$. The only open sets containing either of these
points are
\[
\mathbb{A}
\qquad\text{and}\qquad
\mathbb{X},
\]
and both contain $k$ and $l$. Hence no open set contains exactly
one of these two points. Therefore
$(\mathbb{X},\tau)$ is not $T_{0}$.
\end{proof}

\begin{corollary}
\label{cor:infinite-Hstar-separation}
For the space in Example~\ref{ex:infinite-discrete-Hstar},
$(\mathbb{X},\tau_{H^{*}})$ is $T_{0}$, $T_{1}$, Hausdorff,
regular, completely regular, and normal.
\end{corollary}

\begin{proof}
Since
\[
\tau_{H^{*}}=\mathcal{P}(\mathbb{X}),
\]
the associated space is discrete, and the conclusion follows.
\end{proof}

\begin{remark}
\label{rem:infinite-Hstar-separation}
By the correspondence established in
Section~\ref{sec:separation} and
Section~\ref{sec:higher-separation}, the space
$(\mathbb{X},\tau)$ in
Example~\ref{ex:infinite-discrete-Hstar} satisfies the corresponding
$H^{*}$-separation conditions, although its original topology is not
even $T_{0}$.
\end{remark}

\subsection{Comparison with the Original Topology}
\label{subsec:comparison-original-Hstar}

The preceding examples motivate a general comparison between the
original topology $\tau$ and the associated topology
$\tau_{H^{*}}$.

\begin{proposition}
\label{prop:tau-subset-Hstar}
Suppose that every open subset of $(\mathbb{X},\tau)$ is
$H^{*}$-open and that $H^{*}O(\mathbb{X})$ forms an ordinary
topology. Then
\[
\tau\subseteq\tau_{H^{*}}.
\]
\end{proposition}

\begin{proof}
Let $\mathbb{U}\in\tau$. By hypothesis,
$\mathbb{U}\in H^{*}O(\mathbb{X})$. Since
\[
\tau_{H^{*}}=H^{*}O(\mathbb{X}),
\]
we have
\[
\mathbb{U}\in\tau_{H^{*}}.
\]
Therefore,
\[
\tau\subseteq\tau_{H^{*}}.
\]
\end{proof}

\begin{remark}
\label{rem:Hstar-finer}
Under the hypotheses of
Proposition~\ref{prop:tau-subset-Hstar}, the associated topology
$\tau_{H^{*}}$ is finer than the original topology $\tau$.
The inclusion may be strict. Indeed, in
Examples~\ref{ex:finite-discrete-Hstar} and
\ref{ex:infinite-discrete-Hstar},
\[
\tau\subsetneq\tau_{H^{*}}
=
\mathcal{P}(\mathbb{X}).
\]
Thus the associated $H^{*}$-topology may contain substantially more
open sets than the original topology.
\end{remark}

\begin{proposition}
\label{prop:separation-transfer-original}
Assume that the associated topology $\tau_{H^{*}}$ exists and that
\[
\tau\subseteq\tau_{H^{*}}.
\]
Then:
\[
(\mathbb{X},\tau)\text{ is }T_{0}
\quad\Longrightarrow\quad
(\mathbb{X},\tau)\text{ is }H^{*}\text{-}T_{0},
\]
\[
(\mathbb{X},\tau)\text{ is }T_{1}
\quad\Longrightarrow\quad
(\mathbb{X},\tau)\text{ is }H^{*}\text{-}T_{1},
\]
and
\[
(\mathbb{X},\tau)\text{ is }T_{2}
\quad\Longrightarrow\quad
(\mathbb{X},\tau)\text{ is }H^{*}\text{-}T_{2}.
\]
\end{proposition}

\begin{proof}
Suppose that $(\mathbb{X},\tau)$ satisfies one of the separation
axioms $T_{0}$, $T_{1}$, or $T_{2}$. The open sets required by the
corresponding separation condition belong to $\tau$. Since
\[
\tau\subseteq\tau_{H^{*}}
=
H^{*}O(\mathbb{X}),
\]
the same sets are $H^{*}$-open. Hence the corresponding
$H^{*}$-separation condition follows.
\end{proof}

\begin{remark}
\label{rem:converse-separation-fails}
None of the converse implications in
Proposition~\ref{prop:separation-transfer-original} holds in
general. Indeed, in both
Examples~\ref{ex:finite-discrete-Hstar} and
\ref{ex:infinite-discrete-Hstar}, the original topology is not
$T_{0}$, whereas
\[
\tau_{H^{*}}=\mathcal{P}(\mathbb{X}).
\]
Hence, by
Corollary~\ref{cor:Hstar-Ti-associated}, the corresponding
$H^{*}$-structure is $H^{*}$-$T_{0}$,
$H^{*}$-$T_{1}$, and $H^{*}$-$T_{2}$.
\end{remark}

\begin{remark}
\label{rem:finer-topology-limitation}
The inclusion
\[
\tau\subseteq\tau_{H^{*}}
\]
should not be used to infer analogous implications for arbitrary
higher separation properties. In particular, regularity, complete
regularity, and normality are not transferred merely from the fact
that one topology is finer than another. Such properties must be
verified independently or through the associated-topology
correspondences established in
Section~\ref{sec:higher-separation}.
\end{remark}

\section{Compactness in the $H^{*}$-Generated Structure}
\label{sec:compactness}

We now study compactness with respect to covers by $H^{*}$-open
sets. When the associated topology exists, this notion admits a
direct interpretation in terms of ordinary compactness in
$(\mathbb{X},\tau_{H^{*}})$.

\begin{definition}
\label{def:Hstar-compact}
A subset $\mathbb{K}\subseteq\mathbb{X}$ is said to be
\emph{$H^{*}$-compact} if every cover of $\mathbb{K}$ by
$H^{*}$-open subsets of $\mathbb{X}$ admits a finite subcover.

In particular, $(\mathbb{X},\tau)$ is called
\emph{$H^{*}$-compact} if $\mathbb{X}$ is $H^{*}$-compact.
\end{definition}

The next result identifies this notion with ordinary compactness
relative to the associated topology.

\begin{theorem}
\label{thm:Hstar-compact-associated}
Suppose that
\[
\tau_{H^{*}}=H^{*}O(\mathbb{X})
\]
is an ordinary topology on $\mathbb{X}$. Then, for every
$\mathbb{K}\subseteq\mathbb{X}$,
\[
\mathbb{K}\text{ is $H^{*}$-compact}
\quad\Longleftrightarrow\quad
\mathbb{K}\text{ is compact as a subspace of }
(\mathbb{X},\tau_{H^{*}}).
\]
\end{theorem}

\begin{proof}
Suppose first that $\mathbb{K}$ is $H^{*}$-compact. Let
\[
\{\mathbb{U}_{\lambda}:\lambda\in\Lambda\}
\]
be a family of $\tau_{H^{*}}$-open subsets of $\mathbb{X}$ covering
$\mathbb{K}$. Since
\[
\tau_{H^{*}}=H^{*}O(\mathbb{X}),
\]
every $\mathbb{U}_{\lambda}$ is $H^{*}$-open. By
$H^{*}$-compactness, there exist
$\lambda_{1},\ldots,\lambda_{n}\in\Lambda$ such that
\[
\mathbb{K}
\subseteq
\bigcup_{i=1}^{n}\mathbb{U}_{\lambda_i}.
\]
Hence $\mathbb{K}$ is compact in the subspace topology inherited
from $\tau_{H^{*}}$.

Conversely, suppose that $\mathbb{K}$ is compact as a subspace of
$(\mathbb{X},\tau_{H^{*}})$. Let
\[
\{\mathbb{V}_{\mu}:\mu\in M\}
\]
be an $H^{*}$-open cover of $\mathbb{K}$. Since every
$H^{*}$-open set belongs to $\tau_{H^{*}}$, this is a
$\tau_{H^{*}}$-open cover of $\mathbb{K}$. Compactness yields a
finite subcover. Therefore $\mathbb{K}$ is $H^{*}$-compact.
\end{proof}

\begin{corollary}
\label{cor:Hstar-space-compact-associated}
Suppose that $\tau_{H^{*}}$ is an ordinary topology. Then
\[
(\mathbb{X},\tau)\text{ is $H^{*}$-compact}
\quad\Longleftrightarrow\quad
(\mathbb{X},\tau_{H^{*}})\text{ is compact}.
\]
\end{corollary}

\subsection{Compact Sets in $H^{*}$-Hausdorff Spaces}
\label{subsec:Hstar-compact-Hausdorff}

The preceding correspondence yields the $H^{*}$-analogue of the
classical theorem that compact subsets of Hausdorff spaces are
closed.

\begin{theorem}
\label{thm:Hstar-compact-closed}
Suppose that
\[
\tau_{H^{*}}=H^{*}O(\mathbb{X})
\]
is an ordinary topology. If $(\mathbb{X},\tau)$ is
$H^{*}$-$T_{2}$ and
$\mathbb{K}\subseteq\mathbb{X}$ is $H^{*}$-compact, then
$\mathbb{K}$ is $H^{*}$-closed.
\end{theorem}

\begin{proof}
By Theorem~\ref{thm:Hstar-T2-associated},
$(\mathbb{X},\tau_{H^{*}})$ is Hausdorff. Moreover, by
Theorem~\ref{thm:Hstar-compact-associated}, $\mathbb{K}$ is compact
in $(\mathbb{X},\tau_{H^{*}})$.

Every compact subset of a Hausdorff space is closed. Hence
$\mathbb{K}$ is $\tau_{H^{*}}$-closed. Since the closed subsets of
$(\mathbb{X},\tau_{H^{*}})$ are precisely the $H^{*}$-closed
subsets of $\mathbb{X}$,
\[
\mathbb{K}\in H^{*}C(\mathbb{X}).
\]
Therefore $\mathbb{K}$ is $H^{*}$-closed.
\end{proof}

\begin{corollary}
\label{cor:Hstar-compact-singletons}
Suppose that $\tau_{H^{*}}$ is an ordinary topology and
$(\mathbb{X},\tau)$ is $H^{*}$-$T_{2}$. Then every finite subset
of $\mathbb{X}$ is $H^{*}$-closed.
\end{corollary}

\begin{proof}
Every finite subset of a topological space is compact. Therefore,
by Theorem~\ref{thm:Hstar-compact-associated}, every finite subset
of $\mathbb{X}$ is $H^{*}$-compact. The conclusion follows from
Theorem~\ref{thm:Hstar-compact-closed}.
\end{proof}

\subsection{Compactness and the Original Topology}
\label{subsec:Hstar-compact-original}

We next compare $H^{*}$-compactness with compactness in the original
topology. The direction of implication is determined by the
inclusion relation between $\tau$ and $\tau_{H^{*}}$.

\begin{proposition}
\label{prop:Hstar-compact-implies-original}
Suppose that the associated topology $\tau_{H^{*}}$ exists and
\[
\tau\subseteq\tau_{H^{*}}.
\]
If $\mathbb{K}\subseteq\mathbb{X}$ is $H^{*}$-compact, then
$\mathbb{K}$ is compact with respect to the original topology
$\tau$.
\end{proposition}

\begin{proof}
Let
\[
\{\mathbb{U}_{\lambda}:\lambda\in\Lambda\}
\]
be a $\tau$-open cover of $\mathbb{K}$. Since
\[
\tau\subseteq\tau_{H^{*}}
=
H^{*}O(\mathbb{X}),
\]
every $\mathbb{U}_{\lambda}$ is $H^{*}$-open. Thus the given family
is an $H^{*}$-open cover of $\mathbb{K}$.

By $H^{*}$-compactness, it admits a finite subcover. Hence
$\mathbb{K}$ is compact in the topology $\tau$.
\end{proof}

\begin{remark}
\label{rem:compactness-converse}
The converse of
Proposition~\ref{prop:Hstar-compact-implies-original} does not
follow merely from
\[
\tau\subseteq\tau_{H^{*}}.
\]
Passing from $\tau$ to the finer topology $\tau_{H^{*}}$ may
introduce additional open covers. Consequently, compactness with
respect to $\tau$ does not in general imply compactness with respect
to $\tau_{H^{*}}$.
\end{remark}

\begin{proposition}
\label{prop:equal-topologies-compactness}
If
\[
\tau=\tau_{H^{*}},
\]
then, for every $\mathbb{K}\subseteq\mathbb{X}$,
\[
\mathbb{K}\text{ is compact in }(\mathbb{X},\tau)
\quad\Longleftrightarrow\quad
\mathbb{K}\text{ is $H^{*}$-compact}.
\]
\end{proposition}

\begin{proof}
If
\[
\tau=\tau_{H^{*}}=H^{*}O(\mathbb{X}),
\]
then the $\tau$-open covers of $\mathbb{K}$ are exactly its
$H^{*}$-open covers. Hence the two compactness conditions coincide.
\end{proof}

\subsection{The Discrete Case}
\label{subsec:Hstar-discrete-compactness}

The examples in Section~\ref{sec:examples} also illustrate the
behaviour of $H^{*}$-compactness when the associated topology is
discrete.

\begin{proposition}
\label{prop:discrete-Hstar-compact}
Suppose that
\[
\tau_{H^{*}}=\mathcal{P}(\mathbb{X}).
\]
Then, for every $\mathbb{K}\subseteq\mathbb{X}$,
\[
\mathbb{K}\text{ is $H^{*}$-compact}
\quad\Longleftrightarrow\quad
\mathbb{K}\text{ is finite}.
\]
\end{proposition}

\begin{proof}
Since
\[
\tau_{H^{*}}=\mathcal{P}(\mathbb{X}),
\]
the associated space is discrete. By
Theorem~\ref{thm:Hstar-compact-associated},
$H^{*}$-compactness of $\mathbb{K}$ is equivalent to compactness of
$\mathbb{K}$ in the associated discrete topology.

If $\mathbb{K}$ is finite, then it is compact.

Conversely, suppose that $\mathbb{K}$ is infinite. The family
\[
\bigl\{\{k\}:k\in\mathbb{K}\bigr\}
\]
is an open cover of $\mathbb{K}$ in the discrete topology. No finite
subfamily covers $\mathbb{K}$. Thus $\mathbb{K}$ is not compact and
hence is not $H^{*}$-compact.

Therefore,
\[
\mathbb{K}\text{ is $H^{*}$-compact}
\quad\Longleftrightarrow\quad
\mathbb{K}\text{ is finite}.
\]
\end{proof}

\begin{corollary}
\label{cor:infinite-discrete-not-Hstar-compact}
If $\mathbb{X}$ is infinite and
\[
\tau_{H^{*}}=\mathcal{P}(\mathbb{X}),
\]
then $(\mathbb{X},\tau)$ is not $H^{*}$-compact.
\end{corollary}

\begin{proof}
By Proposition~\ref{prop:discrete-Hstar-compact},
$\mathbb{X}$ is $H^{*}$-compact if and only if $\mathbb{X}$ is
finite. Hence an infinite $\mathbb{X}$ cannot be
$H^{*}$-compact.
\end{proof}

\begin{remark}
\label{rem:compactness-separation-tradeoff}
The discrete examples illustrate that the passage to the associated
$H^{*}$-topology may have different effects on separation and
compactness. In the examples of Section~\ref{sec:examples}, the
associated topology is discrete and therefore satisfies strong
separation properties. If the underlying set is infinite, however,
the same discrete associated topology is not compact. Thus improved
separation behaviour of the associated topology does not imply
compactness.
\end{remark}

\section{Mappings and Preservation Properties}
\label{sec:mappings}

We now examine mappings between spaces equipped with
$H^{*}$-generated structures. The main purpose of this section is to
interpret $H^{*}$-irresolute mappings through the associated
topologies and to derive the corresponding preservation properties.

Let $(\mathbb{X},\tau)$ and $(\mathbb{Y},\sigma)$ be topological
spaces. Their families of $H^{*}$-open sets are denoted by
\[
H^{*}O(\mathbb{X})
\qquad\text{and}\qquad
H^{*}O(\mathbb{Y}),
\]
respectively.

\subsection{$H^{*}$-Irresolute Mappings}
\label{subsec:Hstar-irresolute}

We begin with the notion of $H^{*}$-irresoluteness.

\begin{definition}
\label{def:Hstar-irresolute}
A mapping
\[
f:(\mathbb{X},\tau)\longrightarrow(\mathbb{Y},\sigma)
\]
is called \emph{$H^{*}$-irresolute} if
\[
f^{-1}(\mathbb{V})\in H^{*}O(\mathbb{X})
\]
for every
\[
\mathbb{V}\in H^{*}O(\mathbb{Y}).
\]
\end{definition}

When the associated $H^{*}$-topologies exist, this generalized
mapping condition has an exact classical interpretation.

\begin{theorem}
\label{thm:Hstar-irresolute-continuous}
Suppose that
\[
\tau_{H^{*}}=H^{*}O(\mathbb{X})
\qquad\text{and}\qquad
\sigma_{H^{*}}=H^{*}O(\mathbb{Y})
\]
are ordinary topologies. Then a mapping
\[
f:(\mathbb{X},\tau)\longrightarrow(\mathbb{Y},\sigma)
\]
is $H^{*}$-irresolute if and only if
\[
f:(\mathbb{X},\tau_{H^{*}})
\longrightarrow
(\mathbb{Y},\sigma_{H^{*}})
\]
is continuous.
\end{theorem}

\begin{proof}
Suppose that $f$ is $H^{*}$-irresolute. Let
\[
\mathbb{V}\in\sigma_{H^{*}}.
\]
Since
\[
\sigma_{H^{*}}=H^{*}O(\mathbb{Y}),
\]
the set $\mathbb{V}$ is $H^{*}$-open in $\mathbb{Y}$. Therefore,
by $H^{*}$-irresoluteness,
\[
f^{-1}(\mathbb{V})\in H^{*}O(\mathbb{X})
=\tau_{H^{*}}.
\]
Hence
\[
f:(\mathbb{X},\tau_{H^{*}})
\longrightarrow
(\mathbb{Y},\sigma_{H^{*}})
\]
is continuous.

Conversely, suppose that $f$ is continuous between the associated
topologies. If
\[
\mathbb{V}\in H^{*}O(\mathbb{Y}),
\]
then
\[
\mathbb{V}\in\sigma_{H^{*}}.
\]
Thus
\[
f^{-1}(\mathbb{V})\in\tau_{H^{*}}
=H^{*}O(\mathbb{X}),
\]
and hence $f$ is $H^{*}$-irresolute.
\end{proof}

\begin{remark}
\label{rem:Hstar-irresolute-associated}
Thus, once the associated topologies exist,
$H^{*}$-irresoluteness is not a separate mapping condition at the
associated-topology level; it is precisely ordinary continuity
between the spaces
\[
(\mathbb{X},\tau_{H^{*}})
\qquad\text{and}\qquad
(\mathbb{Y},\sigma_{H^{*}}).
\]
\end{remark}

\subsection{Inverse Images of $H^{*}$-Difference Sets}
\label{subsec:inverse-Hstar-D}

We next recall the inverse-image behaviour of $H^{*}$-difference
sets. This result is included because it will be used in the
preservation of $H^{*}$-difference separation properties.

\begin{theorem}
\label{thm:inverse-Hstar-D}
Let
\[
f:(\mathbb{X},\tau)\longrightarrow(\mathbb{Y},\sigma)
\]
be an $H^{*}$-irresolute surjection. If $\mathbb{D}$ is an
$H^{*}D$-set in $\mathbb{Y}$, then
$f^{-1}(\mathbb{D})$ is an $H^{*}D$-set in $\mathbb{X}$.
\end{theorem}

\begin{proof}
Let $\mathbb{D}$ be an $H^{*}D$-set in $\mathbb{Y}$. Then there
exist
\[
\mathbb{U},\mathbb{V}\in H^{*}O(\mathbb{Y})
\]
such that
\[
\mathbb{D}=\mathbb{U}\setminus\mathbb{V}
\qquad\text{and}\qquad
\mathbb{U}\neq\mathbb{Y}.
\]
Since $f$ is $H^{*}$-irresolute,
\[
f^{-1}(\mathbb{U}),
\,
f^{-1}(\mathbb{V})
\in H^{*}O(\mathbb{X}).
\]
Furthermore,
\[
\begin{aligned}
f^{-1}(\mathbb{D})
&=
f^{-1}(\mathbb{U}\setminus\mathbb{V})\\
&=
f^{-1}(\mathbb{U})\setminus f^{-1}(\mathbb{V}).
\end{aligned}
\]

It remains to verify the properness condition. Since
$\mathbb{U}\neq\mathbb{Y}$, choose
\[
y\in\mathbb{Y}\setminus\mathbb{U}.
\]
Because $f$ is surjective, there exists $x\in\mathbb{X}$ such that
\[
f(x)=y.
\]
Hence
\[
x\notin f^{-1}(\mathbb{U}),
\]
so that
\[
f^{-1}(\mathbb{U})\neq\mathbb{X}.
\]
Therefore $f^{-1}(\mathbb{D})$ is an $H^{*}D$-set in
$\mathbb{X}$.
\end{proof}

\begin{remark}
\label{rem:surjectivity-Hstar-D}
The surjectivity assumption in
Theorem~\ref{thm:inverse-Hstar-D} is essential for the preceding
argument. Indeed, from
\[
\mathbb{U}\neq\mathbb{Y}
\]
alone one cannot conclude that
\[
f^{-1}(\mathbb{U})\neq\mathbb{X}.
\]
Surjectivity ensures that a point of
$\mathbb{Y}\setminus\mathbb{U}$ has a preimage in $\mathbb{X}$.
\end{remark}

\subsection{Preservation of $H^{*}$-Difference Separation}
\label{subsec:Hstar-D1-mappings}

The preceding inverse-image result yields the following preservation
properties.

\begin{theorem}
\label{thm:Hstar-D1-injective}
Let
\[
f:(\mathbb{X},\tau)\longrightarrow(\mathbb{Y},\sigma)
\]
be an $H^{*}$-irresolute bijection. If
$(\mathbb{Y},\sigma)$ is $H^{*}$-$D_{1}$, then
$(\mathbb{X},\tau)$ is $H^{*}$-$D_{1}$.
\end{theorem}

\begin{proof}
Let $k,l\in\mathbb{X}$ with $k\neq l$. Since $f$ is injective,
\[
f(k)\neq f(l).
\]
As $(\mathbb{Y},\sigma)$ is $H^{*}$-$D_{1}$, there exist
$H^{*}D$-sets $\mathbb{D}_{1}$ and $\mathbb{D}_{2}$ in
$\mathbb{Y}$ such that
\[
f(k)\in\mathbb{D}_{1},
\qquad
f(l)\notin\mathbb{D}_{1},
\]
and
\[
f(l)\in\mathbb{D}_{2},
\qquad
f(k)\notin\mathbb{D}_{2}.
\]

Since $f$ is bijective, it is in particular surjective. Hence
Theorem~\ref{thm:inverse-Hstar-D} implies that
\[
f^{-1}(\mathbb{D}_{1})
\quad\text{and}\quad
f^{-1}(\mathbb{D}_{2})
\]
are $H^{*}D$-sets in $\mathbb{X}$. Moreover,
\[
k\in f^{-1}(\mathbb{D}_{1}),
\qquad
l\notin f^{-1}(\mathbb{D}_{1}),
\]
and
\[
l\in f^{-1}(\mathbb{D}_{2}),
\qquad
k\notin f^{-1}(\mathbb{D}_{2}).
\]
Therefore $(\mathbb{X},\tau)$ is $H^{*}$-$D_{1}$.
\end{proof}

The corresponding statement holds for $H^{*}$-$D_{2}$ separation.

\begin{theorem}
\label{thm:Hstar-D2-injective}
Let
\[
f:(\mathbb{X},\tau)\longrightarrow(\mathbb{Y},\sigma)
\]
be an $H^{*}$-irresolute bijection. If
$(\mathbb{Y},\sigma)$ is $H^{*}$-$D_{2}$, then
$(\mathbb{X},\tau)$ is $H^{*}$-$D_{2}$.
\end{theorem}

\begin{proof}
Let $k,l\in\mathbb{X}$ with $k\neq l$. Since $f$ is injective,
\[
f(k)\neq f(l).
\]
By the $H^{*}$-$D_{2}$ property of $\mathbb{Y}$, there exist
disjoint $H^{*}D$-sets $\mathbb{D}_{1}$ and $\mathbb{D}_{2}$ such
that
\[
f(k)\in\mathbb{D}_{1},
\qquad
f(l)\in\mathbb{D}_{2}.
\]

Since $f$ is surjective,
Theorem~\ref{thm:inverse-Hstar-D} shows that
$f^{-1}(\mathbb{D}_{1})$ and
$f^{-1}(\mathbb{D}_{2})$ are $H^{*}D$-sets in $\mathbb{X}$.
Furthermore,
\[
\begin{aligned}
f^{-1}(\mathbb{D}_{1})
\cap
f^{-1}(\mathbb{D}_{2})
&=
f^{-1}(\mathbb{D}_{1}\cap\mathbb{D}_{2})\\
&=
f^{-1}(\varnothing)\\
&=
\varnothing.
\end{aligned}
\]
Also,
\[
k\in f^{-1}(\mathbb{D}_{1}),
\qquad
l\in f^{-1}(\mathbb{D}_{2}).
\]
Therefore $(\mathbb{X},\tau)$ is $H^{*}$-$D_{2}$.
\end{proof}

\subsection{Preservation of $H^{*}$-Compactness}
\label{subsec:Hstar-compact-mappings}

We now turn to a preservation result that follows naturally from
the associated-topology interpretation.

\begin{theorem}
\label{thm:Hstar-compact-image}
Suppose that
\[
\tau_{H^{*}}=H^{*}O(\mathbb{X})
\qquad\text{and}\qquad
\sigma_{H^{*}}=H^{*}O(\mathbb{Y})
\]
are ordinary topologies. Let
\[
f:(\mathbb{X},\tau)\longrightarrow(\mathbb{Y},\sigma)
\]
be $H^{*}$-irresolute. If
$\mathbb{K}\subseteq\mathbb{X}$ is $H^{*}$-compact, then
$f(\mathbb{K})$ is $H^{*}$-compact in $\mathbb{Y}$.
\end{theorem}

\begin{proof}
By Theorem~\ref{thm:Hstar-irresolute-continuous},
\[
f:(\mathbb{X},\tau_{H^{*}})
\longrightarrow
(\mathbb{Y},\sigma_{H^{*}})
\]
is continuous.

By Theorem~\ref{thm:Hstar-compact-associated},
$\mathbb{K}$ is compact as a subspace of
$(\mathbb{X},\tau_{H^{*}})$. Since continuous images of compact
spaces are compact,
\[
f(\mathbb{K})
\]
is compact as a subspace of
$(\mathbb{Y},\sigma_{H^{*}})$.

Applying Theorem~\ref{thm:Hstar-compact-associated} to
$(\mathbb{Y},\sigma)$, we conclude that $f(\mathbb{K})$ is
$H^{*}$-compact.
\end{proof}

\begin{corollary}
\label{cor:Hstar-compact-surjection}
Suppose that the associated topologies $\tau_{H^{*}}$ and
$\sigma_{H^{*}}$ exist. Let
\[
f:(\mathbb{X},\tau)\longrightarrow(\mathbb{Y},\sigma)
\]
be an $H^{*}$-irresolute surjection. If
$(\mathbb{X},\tau)$ is $H^{*}$-compact, then
$(\mathbb{Y},\sigma)$ is $H^{*}$-compact.
\end{corollary}

\begin{proof}
Apply Theorem~\ref{thm:Hstar-compact-image} with
\[
\mathbb{K}=\mathbb{X}.
\]
Since $f$ is surjective,
\[
f(\mathbb{X})=\mathbb{Y}.
\]
Hence $\mathbb{Y}$ is $H^{*}$-compact.
\end{proof}

\subsection{$H^{*}$-Homeomorphisms}
\label{subsec:Hstar-homeomorphisms}

\begin{definition}
\label{def:Hstar-homeomorphism}
A bijection
\[
f:(\mathbb{X},\tau)\longrightarrow(\mathbb{Y},\sigma)
\]
is called an \emph{$H^{*}$-homeomorphism} if both $f$ and
$f^{-1}$ are $H^{*}$-irresolute.
\end{definition}

\begin{theorem}
\label{thm:Hstar-homeomorphism-associated}
Suppose that
\[
\tau_{H^{*}}=H^{*}O(\mathbb{X})
\qquad\text{and}\qquad
\sigma_{H^{*}}=H^{*}O(\mathbb{Y})
\]
are ordinary topologies. Then a bijection
\[
f:(\mathbb{X},\tau)\longrightarrow(\mathbb{Y},\sigma)
\]
is an $H^{*}$-homeomorphism if and only if
\[
f:(\mathbb{X},\tau_{H^{*}})
\longrightarrow
(\mathbb{Y},\sigma_{H^{*}})
\]
is a homeomorphism.
\end{theorem}

\begin{proof}
By Theorem~\ref{thm:Hstar-irresolute-continuous},
$f$ is $H^{*}$-irresolute if and only if
\[
f:(\mathbb{X},\tau_{H^{*}})
\longrightarrow
(\mathbb{Y},\sigma_{H^{*}})
\]
is continuous.

Applying the same theorem to the inverse bijection
\[
f^{-1}:(\mathbb{Y},\sigma)
\longrightarrow
(\mathbb{X},\tau),
\]
we obtain that $f^{-1}$ is $H^{*}$-irresolute if and only if
\[
f^{-1}:(\mathbb{Y},\sigma_{H^{*}})
\longrightarrow
(\mathbb{X},\tau_{H^{*}})
\]
is continuous.

Thus both $f$ and $f^{-1}$ are $H^{*}$-irresolute if and only if
$f$ is a homeomorphism between the associated spaces.
\end{proof}

\begin{corollary}
\label{cor:Hstar-homeomorphic-invariants}
Suppose that the associated topologies exist. Then every
topological invariant of the associated spaces is preserved under
an $H^{*}$-homeomorphism. In particular, the $H^{*}$-compactness
and $H^{*}$-separation properties identified in the preceding
sections are preserved under $H^{*}$-homeomorphisms.
\end{corollary}

\begin{proof}
By Theorem~\ref{thm:Hstar-homeomorphism-associated}, an
$H^{*}$-homeomorphism is precisely a homeomorphism between the
associated spaces. Hence every topological invariant of those
spaces is preserved.

The assertions concerning $H^{*}$-compactness and the
$H^{*}$-separation properties follow from their corresponding
associated-topology characterizations established in
Sections~\ref{sec:separation}--\ref{sec:compactness}.
\end{proof}

\begin{remark}
\label{rem:mapping-associated-view}
The associated-topology approach therefore provides a unified
interpretation of the mapping theory. Whenever
\[
\tau_{H^{*}}=H^{*}O(\mathbb{X})
\quad\text{and}\quad
\sigma_{H^{*}}=H^{*}O(\mathbb{Y})
\]
are ordinary topologies, $H^{*}$-irresolute mappings are precisely
continuous mappings between the associated spaces, while
$H^{*}$-homeomorphisms are precisely homeomorphisms between them.
Consequently, classical preservation results may be applied whenever
the corresponding $H^{*}$-property has already been identified with
a topological property of the associated space.
\end{remark}

\section{Relation with Generalized Topological Structures}
\label{sec:comparison}

The purpose of this section is to place the $H^{*}$-generated
structure developed in the preceding sections within the broader
framework of generalized topology. This distinction is important
because a family of generalized open sets need not satisfy the
finite-intersection axiom of an ordinary topology, although it may
still support meaningful notions of interior, closure, separation,
compactness, and mappings.

Recall that a generalized topology on a nonempty set $\mathbb{X}$
is a family
\[
\Gamma\subseteq\mathcal{P}(\mathbb{X})
\]
such that
\[
\varnothing\in\Gamma
\]
and $\Gamma$ is closed under arbitrary unions. It is called
\emph{strong} if
\[
\mathbb{X}\in\Gamma.
\]

In the present setting, the natural family is
\[
H^{*}O(\mathbb{X}).
\]
As established in
Theorem~\ref{thm:Hstar-strong-generalized-topology}, if
$H^{*}C(\mathbb{X})$ is closed under arbitrary intersections, then
\[
\Gamma_{H^{*}}=H^{*}O(\mathbb{X})
\]
is a strong generalized topology on $\mathbb{X}$.

\subsection{From $H^{*}$-Open Sets to a Generalized Topology}
\label{subsec:Hstar-GT-comparison}

The basic mechanism may be summarized as follows.

\begin{proposition}
\label{prop:Hstar-generalized-mechanism}
Suppose that $H^{*}C(\mathbb{X})$ is closed under arbitrary
intersections. Then
\[
\Gamma_{H^{*}}=H^{*}O(\mathbb{X})
\]
is a strong generalized topology on $\mathbb{X}$, and the
$\Gamma_{H^{*}}$-closed sets are precisely the members of
$H^{*}C(\mathbb{X})$.
\end{proposition}

\begin{proof}
The first assertion is exactly
Theorem~\ref{thm:Hstar-strong-generalized-topology}. By
Proposition~\ref{prop:Gamma-Hstar-closed}, the complements of the
members of $\Gamma_{H^{*}}$ are precisely the members of
$H^{*}C(\mathbb{X})$. Hence these are exactly the
$\Gamma_{H^{*}}$-closed sets.
\end{proof}

Thus two distinct structural levels must be distinguished:
\[
\boxed{
\Gamma_{H^{*}}=H^{*}O(\mathbb{X})
\quad\text{as a strong generalized topology}
}
\]
and, under the additional finite-intersection condition,
\[
\boxed{
\tau_{H^{*}}=H^{*}O(\mathbb{X})
\quad\text{as an ordinary topology}.
}
\]

More precisely, under the first condition,
\[
H^{*}C(\mathbb{X})
\text{ is closed under arbitrary intersections}
\]
implies
\[
\Gamma_{H^{*}}
=
H^{*}O(\mathbb{X})
\text{ is a strong generalized topology}.
\]

If, in addition,
\[
H^{*}C(\mathbb{X})
\text{ is closed under finite unions},
\]
then, by
Theorem~\ref{thm:Hstar-topology-characterization},
\[
\tau_{H^{*}}
=
H^{*}O(\mathbb{X})
\]
is an ordinary topology.

Equivalently, under the standing arbitrary-intersection hypothesis
on $H^{*}C(\mathbb{X})$, the passage from
$\Gamma_{H^{*}}$ to $\tau_{H^{*}}$ occurs precisely when
$H^{*}O(\mathbb{X})$ is closed under finite intersections.

This distinction is essential: finite intersections of
$H^{*}$-open sets are not automatic.

\subsection{The Operator Viewpoint}
\label{subsec:Hstar-operator-viewpoint}

The generalized-topological formulation also gives a natural
interpretation of the $H^{*}$-closure and $H^{*}$-interior
operators. Under the standing assumption that
$H^{*}C(\mathbb{X})$ is closed under arbitrary intersections,
define
\[
c_{H^{*}}:
\mathcal{P}(\mathbb{X})
\longrightarrow
\mathcal{P}(\mathbb{X})
\]
by
\[
c_{H^{*}}(\mathbb{A})
=
H^{*}\operatorname{Cl}(\mathbb{A}),
\]
and
\[
i_{H^{*}}:
\mathcal{P}(\mathbb{X})
\longrightarrow
\mathcal{P}(\mathbb{X})
\]
by
\[
i_{H^{*}}(\mathbb{A})
=
H^{*}\operatorname{Int}(\mathbb{A}).
\]

By the results of Section~\ref{sec:operators}, these operators are,
respectively, extensive and contractive:
\[
\mathbb{A}
\subseteq
c_{H^{*}}(\mathbb{A}),
\qquad
i_{H^{*}}(\mathbb{A})
\subseteq
\mathbb{A}.
\]
They are monotone:
\[
\mathbb{A}\subseteq\mathbb{B}
\Longrightarrow
c_{H^{*}}(\mathbb{A})
\subseteq
c_{H^{*}}(\mathbb{B}),
\]
and
\[
\mathbb{A}\subseteq\mathbb{B}
\Longrightarrow
i_{H^{*}}(\mathbb{A})
\subseteq
i_{H^{*}}(\mathbb{B}),
\]
and they are idempotent:
\[
c_{H^{*}}
\bigl(c_{H^{*}}(\mathbb{A})\bigr)
=
c_{H^{*}}(\mathbb{A}),
\]
\[
i_{H^{*}}
\bigl(i_{H^{*}}(\mathbb{A})\bigr)
=
i_{H^{*}}(\mathbb{A}).
\]

Moreover, the two operators are related by complement duality:
\[
c_{H^{*}}(\mathbb{A})
=
\mathbb{X}\setminus
i_{H^{*}}(\mathbb{X}\setminus\mathbb{A}),
\]
and
\[
i_{H^{*}}(\mathbb{A})
=
\mathbb{X}\setminus
c_{H^{*}}(\mathbb{X}\setminus\mathbb{A}).
\]

Hence the operators
$c_{H^{*}}$ and $i_{H^{*}}$ agree with the closure and interior
operators determined by the strong generalized topology
$\Gamma_{H^{*}}$.

It is important, however, that at this level one should not
automatically impose the finite-set identities characteristic of
ordinary topological closure and interior. Those identities arise
under the stronger conditions ensuring that
$H^{*}O(\mathbb{X})$ is an ordinary topology.

\subsection{Passage to the Associated Topology}
\label{subsec:passage-associated}

Suppose now that the equivalent conditions of
Theorem~\ref{thm:Hstar-topology-characterization} are satisfied.
Then
\[
\tau_{H^{*}}
=
H^{*}O(\mathbb{X})
\]
is an ordinary topology on $\mathbb{X}$.

By Corollary~\ref{cor:Hstar-associated-closure}, for every
$\mathbb{A}\subseteq\mathbb{X}$,
\[
H^{*}\operatorname{Cl}(\mathbb{A})
=
\operatorname{Cl}_{\tau_{H^{*}}}(\mathbb{A}),
\]
and
\[
H^{*}\operatorname{Int}(\mathbb{A})
=
\operatorname{Int}_{\tau_{H^{*}}}(\mathbb{A}).
\]

Thus the following identifications hold:
\[
\begin{array}{ccc}
H^{*}\text{-open sets}
&
\longleftrightarrow
&
\tau_{H^{*}}\text{-open sets},
\\[2mm]
H^{*}\text{-closed sets}
&
\longleftrightarrow
&
\tau_{H^{*}}\text{-closed sets},
\\[2mm]
H^{*}\operatorname{Cl}
&
\longleftrightarrow
&
\operatorname{Cl}_{\tau_{H^{*}}},
\\[2mm]
H^{*}\operatorname{Int}
&
\longleftrightarrow
&
\operatorname{Int}_{\tau_{H^{*}}}.
\end{array}
\]

This correspondence is the structural reason why several notions
defined through the $H^{*}$-open and $H^{*}$-closed sets admit
ordinary topological interpretations in the associated space.

\subsection{Consequences for the $H^{*}$-Structure}
\label{subsec:comparison-separation}

The results of the preceding sections can now be viewed as
consequences of the associated-topology correspondence.

For $i\in\{0,1,2\}$,
\[
(\mathbb{X},\tau)
\text{ is }H^{*}\text{-}T_i
\quad\Longleftrightarrow\quad
(\mathbb{X},\tau_{H^{*}})
\text{ is }T_i.
\]

Similarly, for $i\in\{1,2\}$,
\[
(\mathbb{X},\tau)
\text{ is }H^{*}\text{-}D_i
\quad\Longleftrightarrow\quad
(\mathbb{X},\tau_{H^{*}})
\text{ is }D_i.
\]

The same principle identifies $H^{*}$-regularity,
$H^{*}$-normality, and $H^{*}$-complete regularity with their
classical counterparts in the associated topology. Likewise,
$H^{*}$-compactness becomes ordinary compactness in
$(\mathbb{X},\tau_{H^{*}})$, and an $H^{*}$-irresolute mapping
between two spaces whose associated topologies exist becomes an
ordinary continuous mapping between those associated spaces.

These equivalences should be understood as structural
identifications. They do not assert that the corresponding
properties hold automatically; rather, they show that whenever an
$H^{*}$-property is formulated through the associated
$H^{*}$-open/closed structure, its validity can be examined in the
ordinary topological space $(\mathbb{X},\tau_{H^{*}})$.

\subsection{Comparison with the Original Topology}
\label{subsec:structural-comparison}

The original topology $\tau$ and the associated topology
$\tau_{H^{*}}$ must be kept conceptually distinct. Even when
$\tau_{H^{*}}$ exists, there is no general reason for the two
topologies to coincide.

Under the hypothesis of
Proposition~\ref{prop:tau-subset-Hstar}, one has
\[
\tau\subseteq\tau_{H^{*}},
\]
but this inclusion may be strict.

Indeed, Examples~\ref{ex:finite-discrete-Hstar} and
\ref{ex:infinite-discrete-Hstar} show that it is possible to have
\[
\tau_{H^{*}}
=
\mathcal{P}(\mathbb{X}),
\]
while the original space $(\mathbb{X},\tau)$ fails even to satisfy
the $T_{0}$ separation axiom.

Thus the passage
\[
(\mathbb{X},\tau)
\longmapsto
(\mathbb{X},\tau_{H^{*}})
\]
may substantially alter the topological structure of the underlying
set.

When
\[
\tau\subsetneq\tau_{H^{*}},
\]
the associated topology contains strictly more open sets than the
original topology. This may permit point-separation phenomena in
$(\mathbb{X},\tau_{H^{*}})$ that are absent from
$(\mathbb{X},\tau)$.

However, the presence of a finer associated topology should not be
interpreted as a general improvement of all topological properties.
As observed in Section~\ref{sec:compactness}, for example, passing
to a finer topology may destroy compactness.

\subsection{Scope of the Generalized-Topology Approach}
\label{subsec:scope-GT}

The generalized-topological viewpoint becomes particularly important
when $H^{*}O(\mathbb{X})$ is closed under arbitrary unions but is
not closed under finite intersections. In this situation,
\[
\Gamma_{H^{*}}
=
H^{*}O(\mathbb{X})
\]
is a meaningful strong generalized topology, whereas
\[
\tau_{H^{*}}
=
H^{*}O(\mathbb{X})
\]
cannot be regarded as an ordinary topology.

Thus one obtains the structural distinction
\[
\boxed{
\Gamma_{H^{*}}
\text{ may exist even when }
\tau_{H^{*}}
\text{ does not exist as an ordinary topology}.
}
\]

Consequently, results formulated solely in terms of the strong
generalized topology $\Gamma_{H^{*}}$ and its generalized
closure/interior operators can be investigated without assuming the
finite-intersection axiom. By contrast, results that identify the
$H^{*}$-structure with an ordinary topological structure require the
additional hypotheses ensuring the existence of
$\tau_{H^{*}}$.

This distinction prevents properties of ordinary topologies, or of
other operator-generated generalized topologies, from being
transferred to the present setting without verification.

\begin{remark}
\label{rem:no-automatic-strong-separation}
The existence of the strong generalized topology
$\Gamma_{H^{*}}$ does not by itself imply Hausdorffness,
regularity, complete regularity, normality, perfect normality, or
any other stronger separation property. Similarly, even when
$\tau_{H^{*}}$ exists as an ordinary topology, such properties do
not follow merely from its construction.

Accordingly, stronger separation conclusions established for other
generalized-topological or operator-generated structures require
their own hypotheses and cannot be transferred automatically to the
$H^{*}$-generated structure.
\end{remark}

\section{Conclusion and Future Directions}
\label{sec:conclusion}

In this paper, we investigated the structure induced by
$H^{*}$-open sets from both generalized-topological and ordinary
topological viewpoints. We showed that, under the arbitrary-
intersection closure of $H^{*}C(\mathbb{X})$,
\[
\Gamma_{H^{*}}=H^{*}O(\mathbb{X})
\]
forms a strong generalized topology. Moreover, under this
hypothesis, $\Gamma_{H^{*}}$ becomes an ordinary topology precisely
when $H^{*}O(\mathbb{X})$ is closed under finite intersections,
equivalently, when $H^{*}C(\mathbb{X})$ is closed under finite
unions. In this case, we write
\[
\tau_{H^{*}}=H^{*}O(\mathbb{X}).
\]

When the associated topology exists, the $H^{*}$-closure and
$H^{*}$-interior coincide with the ordinary closure and interior
operators of $(\mathbb{X},\tau_{H^{*}})$. This correspondence
provides a unified interpretation of the $H^{*}$-separation
properties, $H^{*}$-difference separation, regularity, normality,
complete regularity, compactness, and $H^{*}$-irresolute mappings.
The examples further demonstrate that $\tau_{H^{*}}$ may differ
substantially from the original topology $\tau$; in particular, the
associated topology may be discrete even when the original space is
not $T_{0}$.

Further work may focus on characterizing spaces for which
$H^{*}C(\mathbb{X})$ is closed under arbitrary intersections or
finite unions, and on studying properties of
$(\mathbb{X},\Gamma_{H^{*}})$ without assuming the existence of an
ordinary associated topology. Other natural directions include
connectedness, convergence, product constructions, and additional
mapping properties in the $H^{*}$-generated setting.


\bibliographystyle{amsplain}
\bibliography{references}

@article{njȧstad1965some,
  title={On some classes of nearly open sets},
  author={Njastad, Olav},
  journal={Pacific journal of mathematics},
  volume={15},
  number={3},
  pages={961--970},
  year={1965},
  publisher={Mathematical Sciences Publishers}
}

@article{singal1969almost,
  title={On almost regular spaces},
  author={Singal, MK and Arya, S Prabha},
  journal={Glasnik Mat},
  volume={4},
  number={24},
  pages={89--99},
  year={1969}
}

@article{sharma2014xi,
  title={$\xi$-Regular Spaces},
  author={Sharma, MC and Sharma, Poonam and Sharma, Shobha and Singh, Mandeep},
  journal={J. of Applied Sci. and Tech},
  volume={4},
  number={1},
  pages={1--4},
  year={2014}
}

@article{kumar2017xi,
  title={$\xi$-Normal and $\xi$ Regular Spaces in Topological Spaces},
  author={Kumar, Hamant},
  journal={Internat. J. of Sci. Engg. and Res},
  volume={5},
  number={9},
  pages={6--10},
  year={2017}
}

@article{kumar2018softly,
  title={Softly regular spaces in topological spaces},
  author={Kumar, Hamant and Sharma, MC},
  journal={Journal of Emerging Technologies and Innovative Research},
  volume={5},
  number={11},
  pages={183--190},
  year={2018}
}

@article{stone1937applications,
  title={Applications of the theory of Boolean rings to general topology},
  author={Stone, Marshall Harvey},
  journal={Transactions of the American Mathematical Society},
  volume={41},
  number={3},
  pages={375--481},
  year={1937},
  publisher={JSTOR}
}

@article{levine1963semi,
  title={Semi-open sets and semi-continuity in topological spaces},
  author={Levine, Norman},
  journal={The American mathematical monthly},
  volume={70},
  number={1},
  pages={36--41},
  year={1963},
  publisher={Taylor \& Francis}
}

@article{levine1970gen,
  title={Generalized closed sets in topology},
  author={Levine. N},
  journal={Rendiconti del Circolo Matematico di Palermo},
  volume={19},
  pages={89--96},
  year={1970},
  publisher={Springer}
}

@article{arya1974strongly,
  title={On strongly continuous functions},
  author={Arya, SP and Gupta, R},
  journal={Kyungpook Math. J},
  volume={14},
  pages={131--143},
  year={1974}
}

@article{jankovic1985note,
  title={A note on mappings of extremally disconnected spaces},
  author={Jankovi{\'c}, D},
  journal={Acta Mathematica Hungarica},
  volume={46},
  number={1-2},
  pages={83--92},
  year={1985},
  publisher={Akad{\'e}miai Kiad{\'o}, co-published with Springer Science+ Business Media BV~…}
}

@article{hatir1996decomposition,
  title={A decomposition of continuity},
  author={Hat{\i}r, E and Noiri, T and Y{\"u}ksel, S},
   journal={Journal Acta Mathematica Hungarica},
    volume={70},
    pages={145--150},
    year={1996},
  publisher={AKADEMIAI KIADO}
}

@article{rodrigo2012international,
  title={On h-closed set in topologicl spaces},
  author={Rodrigo, J Antony Rex and Jansi, A Hana Selvi and Theodore, Jessie},
 journal={International Journal of Mathematical Archive},
volume={3},
pages={1057--1062},
  year={2012}
}

@article{quasidontchev,
  title={Quasi-normal spaces and $\pi$ g-closed sets},
  author={Dontchev, J and Noiri, T},
  journal={Acta Mathematica Hungarica},
  volume={89},
  number={3},
  pages={211--219},
  year={2000},
  publisher={Akad{\'e}miai Kiad{\'o}, co-published with Springer Science+ Business Media BV~…}
}

@article{decompositiondontchev1996,
  title={On the various decompositions of continuous and some weakly continuous functions},
  author={Dontchev, J and Przemski, M},
  journal={Acta mathematica hungarica},
  volume={71},
  number={1-2},
  pages={109--120},
  year={1996},
  publisher={Akad{\'e}miai Kiad{\'o}, co-published with Springer Science+ Business Media BV~…}
}

@article{Noiri1994,
  author    = {Noiri, T.},
  title     = {Semi-Normal Spaces and Some Functions},
  journal   = {Acta Mathematica Hungarica},
  volume    = {65},
  number    = {3},
  pages     = {305--311},
  year      = {1994}
}

@article{Noiri1998,
  author    = {T. Noiri},
  title     = {Generalized Preclosed Functions},
  journal   = {Mem. Fac. Sci. Kochi Univ. (Math.)},
  volume    = {19},
  pages     = {13--20},
  year      = {1998}
}

@article{sharma2015h,
  title={H*-NORMAL SPACES},
  author={Sharma, Nidhi and Sharma, Poonam and Tomar, Neeraj Kumar},
  journal={Acta Ciencia Indica},
  volume={2},
  number={XLIM},
  pages={105--108},
  year={2015},
  publisher={Acta Ciencia Indica}
}

@article{Sundaram2000,
  title={On a-closed sets in topology},
  author={Sundaram Pand Sheik John M},
  journal={ Acta Ciencia Indica},
  volume={4},
  pages={389--392},
  year={2000}
}

@article{tomarsome2024,
  title={Some Separation Axioms in Topological Spaces},
  author={Tomar Neeraj, M.C. Sharma, A. Ujlayan and Zulfeqarr},
  journal={ The Mathematics Education},
  volume={LVIII No. 3},
  pages={48--60},
  year={Sep. 2024},
note       = {\url{https://doi.org/10.5281/zenodo.14823030}},
}

@article{Tomar2026,
  title   = {$\pi H^{*}$-Closed Sets in Topological Spaces},
  author  = {Tomar, Neeraj K. and Rani, Saroj and Subadra, Nemani},
  journal = {Bulletin of the Society of Parana Mathematics},
  volume  = {44},
  number  = {2},
  pages   = {1--19},
  year    = {2026},
  note       = {\url{https://doi. 10.5269/bspm.80446}},
}

@article{Tong1982,
  title   = {A Separation Axiom Between $T_0$ and $T_1$},
  author  = {Tong, J.},
  journal = {Annales de la Soci{\'e}t{\'e} Scientifique de Bruxelles},
  volume  = {96},
  pages   = {85--90},
  year    = {1982}
}

\end{document}